\documentclass[letterpaper,11pt]{amsart}
\usepackage{amsmath,amssymb,amsthm,pinlabel,tikz,hyperref,mathrsfs,color}
\usepackage{verbatim}
\newcommand{\nc}{\newcommand}
\nc{\dmo}{\DeclareMathOperator}
\dmo{\ra}{\rightarrow}
\dmo{\N}{\mathbb{N}}
\dmo{\Z}{\mathbb{Z}}
\dmo{\Q}{\mathbb{Q}}
\dmo{\R}{\mathbb{R}}
\dmo{\C}{\mathcal{C}}
\dmo{\AC}{\mathcal{AC}}
\dmo{\Mod}{Mod}
\dmo{\PMod}{PMod}
\dmo{\B}{B}
\dmo{\PB}{PB}
\dmo{\I}{\mathcal{I}}
\dmo{\el}{\ell_{\C}}
\dmo{\NN}{\mathcal{N}}
\dmo{\rk}{rk}

\usetikzlibrary{decorations.markings}
\tikzset{->-/.style={decoration={
  markings,
  mark=at position #1 with {\arrow[scale=3]{>}}},postaction={decorate}}}
  
\nc{\nt}{\newtheorem}

\nt{theorem}{Theorem}
\newtheorem{thm}{{\bf Theorem}}[section]
\newtheorem{lem}[thm]{{\bf Lemma}}
\newtheorem{cor}[thm]{{\bf Corollary}}
\newtheorem{prop}[thm]{{\bf Proposition}}

\newtheorem{remark}[thm]{Remark}
\newtheorem{ex}[thm]{Example}

\newtheorem{ques}[thm]{Question}
\newtheorem{question}[thm]{Question}
\newtheorem{conj}[thm]{Conjecture}

\numberwithin{equation}{section}
\nt*{theorem_nonumber}{Theorem}

\title[Pseudo-Anosov monodromies of fibered 3-manifolds]
{Asymptotic
  translation lengths and normal generation for pseudo-Anosov monodromies
  of fibered 3-manifolds}
\date{\today}
\author{Hyungryul Baik}
\address{%
		Department of Mathematical Sciences, KAIST\\
		291 Daehak-ro Yuseong-gu, Daejeon, 34141, South Korea 
}
\email{%
        hrbaik@kaist.ac.kr
}

\author{Eiko Kin}

\address{%
        Department of Mathematics, Graduate School of Science, Osaka University Toyonaka, Osaka 560-0043, JAPAN
}
\email{%
        kin@math.sci.osaka-u.ac.jp
        }

\author{Hyunshik Shin}

\address{%
        Department of Mathematics\\
        University of Georgia\\
		Athens, GA 30602, USA
}
\email{%
        hyunshik.shin@uga.edu
}

\author{Chenxi Wu}

\address{%
        Department of Mathematics, Rutgers University - Busch Campus\\
        	110 Frelinghuysen Road, Piscataway, NJ 08854-8019, USA
}
\email{%
        cwu@math.rutgers.edu
}

\begin{document}
\begin{abstract}
Let $M$ be a hyperbolic fibered 3-manifold. 
We study properties of sequences $(S_{\alpha_n},  \psi_{\alpha_n})$ of fibers and monodromies 
for primitive integral classes in the fibered cone of $M$. 
The main object is the asymptotic translation length $\ell_{\C} (\psi_{\alpha_n})$ of the pseudo-Anosov monodromy $ \psi_{\alpha_n}$ on the curve complex. 
We first show that 
there exists a constant $C>0$ depending only on the fibered cone 
such that for any primitive integral class $(S, \psi)$ in the fibered cone, 
 $\ell_{\C} (\psi)$ is bounded from above by $C/|\chi(S)|$. 
We also obtain a moral connection 
between $\ell_{\C} (\psi)$ and the normal generating property of  $\psi$ in the mapping class group on $S$.  
We show that for all but finitely many primitive integral classes $(S, \psi)$ 
in an arbitrary 2-dimensional slice of the fibered cone, 
$\psi$  normally generates the mapping class group on $S$. 
In the second half of the paper, we study if it is possible to obtain a continuous extension of normalized asymptotic translation lengths on the curve complex as a function on the fibered face. An analogous question for normalized entropy has been answered affirmatively by Fried and the question for normalized asymptotic translation length on the arc complex in the fully punctured case has been answered negatively by Strenner. 
We show that such an extension in the case of the curve complex does not exist in general by explicit computation for sequences in the fibered cone of the magic manifold. 

\end{abstract}

\maketitle

%%%%%%%%%%%%%%%%%%%%%%%%%%%%%%%%%%%%%%%%%%%%%%%%
%
%							Introduction
%
%%%%%%%%%%%%%%%%%%%%%%%%%%%%%%%%%%%%%%%%%%%%%%%%

\section{Introduction}	\label{section:introduction}

Let $M$ be a hyperbolic fibered  3-manifold. 
Thurston introduced the so-called Thurston norm on the first
cohomology group of $M$, and showed that the unit norm ball is a finite sided polyhedron. 
Let $F$ be a top-dimensional face of this polyhedron and consider a primitive
integral class contained in the open cone $ \mathscr{C}= \mathscr{C}_F$ over $F$. 
Thurston showed that if this cohomology class corresponds to
a fibration of $M$ over the circle $S^1$, then all primitive integral classes
in $ \mathscr{C}$  correspond to fibrations of $M$ over $S^1$. 
In such case, we call $F$ a \textit{fibered face} and the open cone  $ \mathscr{C}$  
a \textit{fibered cone}. 
%A face $F$ of this polytope is called a \textit{fibered face} if
%a primitive integral class contained in the cone over $F$ represents
%a fibration of $M$ over $S^1$. Thurston showed that if $F$ is a
%fibered face, then actually all integral primitive classes in the cone
%over $F$ correspond to fibrations of $M$ over $S^1$, which justfies
%the name 'fibered face', and in this case the cone over $F$ is called
%a \textbf{fibered cone}. 
%For each integral cohomology class $\alpha$ in the fibered cone, one
%can consider its fiber $S_\alpha$ and the monodromy
%$\psi_\alpha$. Abusing the notation, sometimes we let the pair
%$(S_\alpha, \psi_\alpha)$ denote the corresponding cohomology class. 
For each primitive integral  class $\alpha \in \mathscr{C}$,
let ($S_{\alpha}, \psi_{\alpha})$ be the pair of corresponding fiber and its monodromy.
Since $M$ is hyperbolic, the monodromy 
$\psi_{\alpha}$ is pseudo-Anosov by Thurston's hyperbolization theorem 
(see, for example \cite[Theorem 13.4]{FarbMargalit12}).
In this paper, we study the asymptotic translation length of $\psi_{\alpha}$ on the curve complex of 
the surface $S_{\alpha}$ and the normal generators of mapping class groups $\mathrm{Mod}(S_{\alpha})$.
%of $S_{\alpha}$ in terms of sequences of primitive integral classes in $\mathscr{C}$.

Let $G$ be a group acting isometrically on a metric space $(X, d_X)$.
For $h \in G$, the \textit{asymptotic} \textit{translation length} (or \textit{stable length}) of $h$  is defined by
$$ \ell_X(h) = \liminf_{n \to \infty} \dfrac{ d_X(x, h^n x) }{n}, $$ 
where $x$ is a point in $X$.	
It is not hard to see that $\ell_X(h)$ is independent of the choice of $x$.
%One can think of $\ell_X(h)$ as how much the iterates of $h$ move $x$ in $X$
%on average. 

For a surface $S$, let $\mathcal{T}(S)$ be the Teichm\"uller space
of $S$ and  let $\mathcal{C}(S)$ be the curve complex of $S$. Since
$\psi_\alpha$ acts by an isometry on both $\mathcal{T}(S_\alpha)$ and
$\mathcal{C}(S_\alpha)$, one can consider the asymptotic
translation lengths of $\psi_\alpha$ on $\mathcal{T}(S_\alpha)$ and
on $\mathcal{C}(S_\alpha)$, denoted by $\ell_\mathcal{T}(\psi_\alpha)$
and $\ell_{\mathcal{C}}(\psi_\alpha)$ respectively. 
%For an element $g$ of $G$, one can consider how much its iterates move points in $x$ on average. Namely, one defines the asymptotic (or stable) translation length of $g$ for its action on $X$ by 
%$$ \ell_X(g) = \lim_{n \to \infty} \dfrac{ d(x, g^nx) }{n}, $$ 
%where $x$ is a point in $X$ and $g^n$ denotes the composition of $g$
%with itself $n$ times. 

There has been a lot of work on
$\ell_{\mathcal{T}}(\psi_\alpha)$ for primitive integral classes $\alpha$ in the fibered cone. 
See 
\cite{FLP, Fried82, fried1982geometry, Matsumoto87, LongOertel97, McMullen00}.

In the case of $\ell_{\mathcal{C}}(\psi_\alpha)$, there has
also been some progress in the literature. 
See 
\cite{MasurMinsky99, Bowditch08, FarbLeiningerMargalit08, GadreTsai11, GadreHironakaKentLeininger13, Valdivia14, AougabTaylor15,  Valdivia17, KinShin18, BaikShin18, BaikShinWu18}.

The following is a general upper bound  of $\ell_{\mathcal{C}}(\psi_\alpha)$ 
in the fibered cone 
in terms of the Euler characteristic $\chi(S_{\alpha})$ of $S_{\alpha}$.

\begin{thm}[\cite{BaikShinWu18}] \label{thm:BaikShinWu}
Let $F$ be a fibered face of a closed
  hyperbolic  fibered 3-manifold $M$. 
  Let $K$ be a 
  compact subset of the interior $int(F)$ of $F$. 
  Then there exists a constant
  $C $ depending on $K$ such that for any sequence $(S_{\alpha_n}, \psi_{\alpha_n})$ of primitive integral classes which is contained
  in the intersection between the cone over $K$ and a  $(d+1)$-dimensional rational subspace of $H^1(M)$, we have
  $$\ell_{\mathcal{C}}(\psi_{\alpha_n}) \leq  \frac{C}{|\chi(S_{\alpha_n})|^{1+\frac{1}{d}} }. $$
\end{thm}

Here $(d+1)$-dimensional rational subspace of $H^1(M)$ means a subspace of $H^1(M)$ 
which admits a basis $v_1, \cdots, v_{d+1} \in H^1(M; {\Bbb Q})$. 
 We note that in \cite{BaikShinWu18} the above theorem was stated in the case 
 of closed hyperbolic fibered 3-manifolds, but almost the same proof can be adopted
 to the case of compact hyperbolic fibered 3-manifolds possibly with boundary, 
 see Remark \ref{rem_BSW}.

Two additional questions naturally arise from Theorem  \ref{thm:BaikShinWu}. 
First, what can we say
if the sequence is not contained in the cone over any compact subset of 
the fibered face $F$?
For instance, given a sequence that has a subsequence converging projectively to the
boundary $\partial F$, can we determine the upper bound of the 
asymptotic translation length of the  pseudo-Anosov monodromies?
We answer the first question in the following theorem.

%\marginal{in theorem 3.1, a closed $\rightarrow$ a compact (E)}
\newtheorem*{thm:fibration}{Theorem \ref{thm:newbound}}
\begin{thm:fibration}
Let $F$ be a fibered face of a compact hyperbolic fibered 3-manifold possibly with boundary.
Then there exists a constant $C$ depending on $F$ such that 
for any primitive integral class $(S, \psi) \in \mathscr{C}_F$, 
we have
$$\ell_{\C} (\psi) \leq \frac{C}{|\chi(S)|}.$$
\end{thm:fibration}

We remark that the upper bound in Theorem \ref{thm:newbound} is optimal. 
In Lemma \ref{lem:optimalupperbound}, we give an explicit sequence $(S_{\alpha_n},  \psi_{\alpha_n})$ 
converging projectively to a  point in   $\partial F$ such that the asymptotic translation 
length of the corresponding pseudo-Anosov monodromy is comparable to
 $1/|\chi(S_{\alpha_n})|$. 
That is, there exists a constant $C$ such that
$$\frac{1}{C} \frac{1}{|\chi(S_{\alpha_n})|} \leq \ell_{\mathcal{C}}(\psi_{\alpha_n}) \leq  \frac{C}{|\chi(S_{\alpha_n})|}.$$
In general, for real-valued functions $A(x)$ and $B(x)$, 
we say that $A(x)$ is \textit{comparable} to $B(x)$ 
%if there existsa universal constant $C$ (i.e., a constant independent of $x$) 
if there exists a constant $C$ independent of $x$ 
such that $1/C \leq A(x)/B(x) \leq C$.
We denote it by $A(x) \asymp B(x)$.

The second question is whether the upper bound in Theorem
\ref{thm:BaikShinWu} is sharp. 
It is noted in \cite{BaikShinWu18} that the bound is optimal for $d = 1$. 
In this paper, we show that it is also optimal when
$d=2$ by constructing an example coming from %a particular hyperbolic fibered $3$-manifold with three cusps, 
%the so called 
the magic manifold $N$, 
which is the exterior of some $3$ components link in the $3$-sphere $S^3$.

%\newtheorem*{thm:magic}{Theorem \ref{thm:sharpbound}}
%\begin{thm:magic}
%In a fibered cone of the magic manifold $N$, 
%there exists a  sequence $(S_{\alpha_i}, \psi_{\alpha_i})$ of primitive integral classes 
%  with $|\chi(S_{\alpha_i})| \to \infty$ as $i \to \infty$ such that  
%  $(S_{\alpha_i},  \psi_{\alpha_i})$ is not contained in any finite union of
%  2-dimensional subspaces of $H^1(N)$, 
%  $(S_{\alpha_i},  \psi_{\alpha_i})$ converges projectively to a point  in the interior of the fibered face as $i \to \infty$, 
%  and 
%$$\ell_{\mathcal{C}}(\psi_{\alpha_i}) \asymp
%  \frac{1}{|\chi(S_{\alpha_i})|^{\frac{3}{2}} }. $$
%\end{thm:magic} 

%Then we greatly generalize the previous theorem and obtain the following result.

\newtheorem*{thm:allrationalexponents}{Theorem \ref{thm:allrational}}
\begin{thm:allrationalexponents}
Let $F$ be a fibered face of the magic manifold. 
Then there exist two points  
$\mathfrak{b}_0 \in \partial F$ and $\mathfrak{c}_0 \in int(F)$ 
 which satisfy the following. 
\begin{enumerate}
\item 
For any $r \in \mathbb{Q} \cap [1,2)$, there exists  a sequence $(S_{\alpha_n},  \psi_{\alpha_n})$ of primitive integral classes 
 in $\mathscr{C}_F$ converging projectively to $\mathfrak{b}_0 $ as $n \to \infty$ 
  such that 
$$\ell_{\mathcal{C}}(\psi_{\alpha_n}) \asymp
  \frac{1}{|\chi(S_{\alpha_n})|^{r} }. $$
%For any $r \in \Q \cup [1, 2)$, there is a sequence of primitive integral classes $\alpha_n$
%approaching the boundary of the cone $C_F$ as $n \ra \infty$  such that $|\chi(S_{\alpha_n})| \ra \infty$
%and $$\ell_{\C}(\psi_{\alpha_n}) \asymp \frac{1}{|\chi(S_{\alpha_n})|^r}.$$

\item 
For any $r \in \Q \cap [\frac{3}{2}, 2]$, 
there exists a sequence $(S_{\alpha_n},  \psi_{\alpha_n})$ of primitive integral classes  in $\mathscr{C}_F$ 
converging projectively to $\mathfrak{c}_0$ as $n \rightarrow \infty$
such that 
$$\ell_{\C}(\psi_{\alpha_n}) \asymp \frac{1}{|\chi(S_{\alpha_n})|^r}.$$
In particular, the upper bound in Theorem  \ref{thm:BaikShinWu} is optimal  when $d=2$. 
\end{enumerate}
\end{thm:allrationalexponents}

%\marginal{I mention Strenner's result on the arc curve complex. 
%Please feel free to modify it (E)}

%As an immediate corollary of Theorem \ref{thm:allrational}(2), we conclude that 
%{\color{red} there is no normalization of the asymptotic translation length function defined on the rational points of the fibered face, 
%which continuously extends to the whole fibered face.
%This is contrary to the result by Fried.
%He proved that the normalized  entropy function of pseudo-Anosov monodromies has a continuous extension
%on the fibered face,
%which is strictly convex.}

As an immediate corollary of Theorem \ref{thm:allrational}, we conclude that 
there is no normalization of the asymptotic translation length function defined on the rational classes of the fibered face, 
which continuously extends to the whole fibered face. 
More precisely, we have the following. 
\newtheorem*{cor:noextension}{Corollary \ref{cor:nocontinuousextension}}
\begin{cor:noextension} 
%Let $\mathscr{C}$ be a fibered cone of a compact hyperbolic fibered $3$-manifold possibly with boundary. 
%Let $H^1_{\mathrm{prim}}(M; {\Bbb Z})$ be a subset of $H^1(M; {\Bbb Z})$ 
%consisting of primitive integral classes. 
%Then there is no normalization of the asymptotic translation
%  length function 
%  \begin{eqnarray*}
%  \mathscr{C} \cap H^1_{\mathrm{prim}}(M; {\Bbb Z}) &\rightarrow& {\Bbb R}_{\ge 0}
%  \\
%(S_{\alpha}, \psi_{\alpha}) &\rightarrow& \ell_{\C}(\psi_{\alpha})
%  \end{eqnarray*}
%  in terms of the Euler characteristic $\chi(S_{\alpha})$ of the fiber 
%  which admits  a continuous extension to $ \mathscr{C} $
Let $F$ be a fibered face of the magic manifold $N$.
For $\alpha \in F \cap H^1(N;\Q)$, let $(S_{\widetilde{\alpha}}, \psi_{\widetilde{\alpha}})$
be the fiber and pseudo-Anosov monodromy corresponding to the primitive integral class $\widetilde{\alpha}$ lying on 
the ray of $\alpha$ passing through the origin.
Then there is no normalization of the asymptotic translation length function
\begin{align*}
F \cap H^1(N;\Q) & \rightarrow \R_{\geq 0}\\
\alpha &\mapsto \ell_{\C}(\psi_{\widetilde{\alpha}}),
\end{align*}
in terms of the Euler characteristic $\chi(S_{\widetilde{\alpha}})$ which admits a continuous extension on $F$.
\end{cor:noextension} 

For the arc complex, 
Strenner defined in \cite{Strenner18}  the normalized asymptotic translation length function $\mu_d$ for each integer $d \ge 1$ 
on the rational classes of a fibered face with the fully punctured condition. 
Strenner  proved in the same paper 
that the functions $\mu_d$ for $d \ge 2$ are typically nowhere continuous. 
His result and Corollary \ref{cor:nocontinuousextension} stand in contrast to Fried's result \cite{Fried82}. 
See also Matsumoto \cite{Matsumoto87} and McMullen \cite{McMullen00}. 
They proved that the normalized  entropy function of pseudo-Anosov monodromies has a continuous extension
on the fibered face, which is strictly convex.

Now we turn our attention to normal generation of mapping class groups.  
Let $S= S_{g,n}$ be an orientable surface of genus $g$ with $n$ punctures, possibly $n=0$.  
We denote $S_{g,0}$ by $S_g$.  
The next theme of this paper is normal generators of the mapping class group $\mathrm{Mod}(S)$.
%that is the group of isotopy classes of orientation-preserving homeomorphisms of $S$ preserving punctures setwise. 
We say that an element $h$ of a group $G$ {\it normally generates} $G$
if the normal closure of $h$ is equal to $G$.
For a given primitive class $(S_\alpha, \psi_\alpha)$ in the fibered cone $\mathscr{C}$, 
when does $\psi_\alpha$ normally generate  $\Mod(S_\alpha)$? 
This question is motivated by the work of
Lanier--Margalit \cite{LanierMargalit18}. 
%They showed that for a pseudo-Anosov element $f$ of
%$\Mod(S)$ for a closed surface $S$, of genus at least 3, 
%if the stretch
%factor $\lambda(f)$ is smaller than $\sqrt{2}$, then $f$ normally
%generates $\Mod(S)$. 
They showed in the same paper that for a pseudo-Anosov element $f \in \Mod(S_g)$, 
if the stretch
factor $\lambda(f)$ is smaller than $\sqrt{2}$, then $f$ normally
generates $\Mod(S_g)$. 

This connects up with our brief discussion about asymptotic
translation length, since the logarithm of the stretch factor $\log \lambda(f)$ is equal to $\ell_\mathcal{T}(f)$. 
In other words, if a pseudo-Anosov element of $\Mod(S)$ is contained in some proper
normal subgroup, then its asymptotic
translation length on the Teichm\"uller space cannot be too small. 
%We conjecture an analogous statement is true for the curve complexes, i.e., 
It is natural to ask an analogous statement for the curve complexes, i.e., 
if a pseudo-Anosov element of $\Mod(S)$ is contained in some proper
normal subgroup, then its asymptotic
translation length on the curve complex cannot be too small in some sense. 
The following question is raised by Dan Margalit \cite{margp}.

\begin{ques}
\label{conj:normalsubgroup} 
For a subgroup $H $  of $\Mod(S_g)$, let us set 
$$L_\mathcal{C}(H) = \min \{ \ell_\mathcal{C}(f) : f \mbox{ is pseudo-Anosov and } f \in H\}.$$ 
Is there a constant $C>0$ such that for any $g \geq 2$ and 
for any proper normal subgroup $H$ of $\Mod(S_g)$, we have
$$L_\mathcal{C}(H) \geq \frac{C}{g}?$$ 
\end{ques} 

As a partial evidence toward this question, 
it is shown by Baik--Shin \cite{BaikShin18} 
that 
$$L_{\C}(\I_g) \asymp \frac{1}{g},$$ 
where $\I_g$ is the Torelli group, i.e., the proper normal subgroup of $\Mod(S_g)$ whose action on
the first homology is trivial. 
In fact, by \cite[Theorem 3.2]{BaikShin18}, we have 
$L_{\C}(\I_g) \ge \frac{1}{96(g-1)}$ for all $g \ge 2$.

Combining with Theorem \ref{thm:newbound}, 
we propose the following conjecture 
regarding the normal generators of mapping class groups contained in the fibered cone 
which was originally asked as a question by Dan Margalit \cite{margp}. 
%\marginal{could you add the reference "personal communication"?} 

%the following conjecture
%regarding the normal generators of mapping class groups contained in the fibered cone
%was proposed by Dan Margalit. 

\begin{conj}
\label{conj:normalgeneration} 
Let $F$ be a fibered face of a closed  hyperbolic fibered 3-manifold $M$. 
Then for all but finitely many primitive classes $(S_\alpha,\psi_\alpha) \in \mathscr{C}_F$, 
$\psi_\alpha$ normally generates $\Mod(S_\alpha)$. 
\end{conj}

We give a partial answer when primitive integral classes are contained in a 2-dimensional rational subspace of $H^1(M)$. 
See also Remark \ref{remark:dim}.

\newtheorem*{thm:normal}{Theorem \ref{thm:normalgeneration}}
\begin{thm:normal}
Let $F$ be a fibered face of a closed  hyperbolic fibered 3-manifold $M$, 
and let $L$ be a 2-dimensional rational subspace of $H^1(M)$.  
Then for all but finitely many primitive integral classes $(S, \psi)$ in $\mathscr{C}_F \cap L$, $\psi$ normally generates $\Mod(S)$. 
In particular, if the rank of $H^1(M)$ equals $2$, then 
Conjecture \ref{conj:normalgeneration} is true.  
\end{thm:normal}

\subsection*{Acknowledgments}
We thank Susumu Hirose, Michael Landry and Dan Margalit for helpful comments. 
The first author was partially supported by Samsung Science \& Technology Foundation grant No. SSTF-BA1702-01. 
The second author was supported by 
Grant-in-Aid for Scientific Research (C) (No. 18K03299), JSPS.  
The third author was supported by Basic Science Research Program
through the National Research Foundation of Korea(NRF) funded by the Ministry of Education
(NRF-2017R1D1A1B03035017).

%%%%%%%%%%%%%%%%%%%%%%%%%%%%%%%%%%%%%%%%%%%%%%%%
%
%						Arithmetic Sequences in the fibered cone
%
%%%%%%%%%%%%%%%%%%%%%%%%%%%%%%%%%%%%%%%%%%%%%%%%

\medskip
\section{Arithmetic sequences in the fibered cone}	\label{sectoin:arithmeticsequences}
For a hyperbolic  3-manifold $M$ possibly with boundary $\partial M$, 
Thurston \cite{Thurston86} defined a norm $|| \cdot ||$ on $H_2(M, \partial M;\R)$. 
It turns out the unit norm ball $B_M$ with respect to the Thurston norm is a finite-sided polyhedron. 
Let $F$ be a  top-dimensional face of $B_M$. 
We consider an open cone $\mathscr{C} = \mathscr{C}_F$ over $F$. 
Thurston showed that if $M$ is a fibered $3$-manifold, then 
either all integral points in $\mathscr{C}$ are fibered or none of them are fibered.
In the former case, we call $\mathscr{C}$ a \textit{fibered cone}. 
We denote by $\overline{\mathscr{C}}$ the closure of the fibered cone $\mathscr{C}$. 

By abuse of notation, the first cohomology classes are treated as their
dual second homology classes throughout this paper without explicitly
mentioning it. 
Furthermore, we will write a primitive integral class $\alpha \in H^1(M)$ as a pair $(S, \psi)$ 
when $S$ and $\psi$ are the fiber and the monodromy for the fibration over $S^1$ 
corresponding to $\alpha$.

In this section, we will show a key property of infinite arithmetic sequences
in a fibered cone for the proof of Theorem \ref{thm:normalgeneration}. 
Here by an arithmetic sequence we mean a sequence of them \
$(\alpha + n \beta)_{n \in \mathbb{Z}_{\geq 0}}$ where $\alpha$ (resp. $\beta$) is a primitive integral class in a fibered cone $\mathscr{C}$ 
(resp. the closure $\overline{\mathscr{C}}$ of the fibered cone $\mathscr{C}$).
We first need to find some criterion
for a given element of the mapping class group to be a normal generator. 
In \cite{LanierMargalit18}, the so-called {\it well-suited curve criterion} is introduced. 
Roughly speaking, this criterion says that if there is a simple closed curve $c$ such that the configuration of $c \cup f(c)$ is simple enough,
then $f$ is a normal generator for the mapping class group.

Here we state one special case that we need and show its proof for the sake of completeness.
For more general statements, see \cite[Sections 2,7,9]{LanierMargalit18}. 
 For a closed curve $c$ in the surface $S_g$ without specified orientation, 
$[c]$ means the homology class in $H_1(S_g)$ with arbitrary orientation. 
%\marginal{Just before Lem 2.1, I add ``For a closed curve $c$ ... the homology class in $H_1(S_g)$ with arbitrary orientation. "}

\begin{lem}[Lemma 2.3 in \cite{LanierMargalit18}]
\label{lem:normalcriterion}
Let $f \in \Mod(S_g)$ for $g \geq 3$.
Suppose that there is a nonseparating curve $c$ in $S_g$ so that 
$c$ and $f(c)$ are disjoint  
and $\pm [c] \neq [f(c)] \in H_1(S_g)$. 
Then  the normal closure of $f$ is $\Mod(S_g)$.
%%
%%
%Let $f \in \Mod(S_g)$ for $g \geq 2$.
%Suppose that there is a nonseparating curve $c$ in $S_g$ so that 
%$c$ and $f(c)$ are disjoint  
%and $[c] \neq [f(c)] \in H_1(S_g)$. Then the normal closure of $f$ contains the commutator subgroup of $\Mod(S_g)$.
%In particular, if $g \geq 3$, then the normal closure of $f$ is $\Mod(S_g)$.
\end{lem}
%\marginal{$[c] \neq [f(c)] \in H_1(S_g) \to \pm [c] \neq [f(c)] \in H_1(S_g)$}

%\marginal{In the statement of lemma 2.1, I replace "$i(c,f(c))$" with "$c$ and $f(c)$ are disjoint", 
%since $i(,)$ is not defined (E)}

\begin{proof} 
Let $f$ and $c$ be as in the statement of the lemma. Then one can find
nonseparating curves $a,b,d,x$, and $y$ which satisfy the following
conditions. 
%(See Figure \ref{fig_lantern}). 

\begin{itemize}
\item $a, b, c$, and $d$ bound a subsurface $S$ of $S_g$ which is homeomorphic to a $4$-punctured
sphere.  
\item each of the triple of curves $(a, b, x)$, $(b,d,y)$ and
  $(b,c,f(c))$ bounds a pair of pants contained in $S$. 
\item no two of the curves $a,b,c,d,x,y$, and $f(c)$ are homologous. 
\end{itemize} 

\begin{center}
\begin{figure}[t]
\includegraphics[scale=.20]{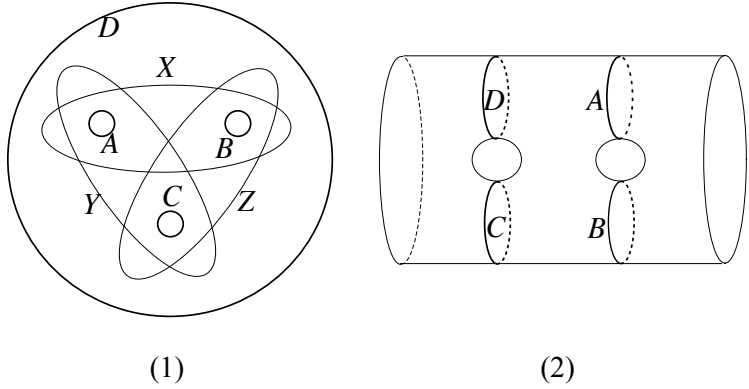}
\caption{(1) $4$-punctured sphere. 
(2) Genus 2 surface with two boundary components.} 
\label{fig_lantern}
\end{figure}
\end{center}

To see the existence of such curves, start with Figure \ref{fig_lantern}(1) 
which is the surface of genus 0 with four boundary components ($4$-punctured sphere) 
labeled by $A, B, C, D$. Glue a pair of pants along the boundary components labeled by $A$ and $B$, and glue another pair of pants along the boundary components labeled 
by $C$ and $D$. Then we get  a surface of genus 2 with two boundary components (Figure \ref{fig_lantern}(2)). 
Along the two boundary components, we glue in another surface of genus $k \geq 0$ with two boundary components. The resulting surface is a closed surface of genus $3+k$. We take $k$ so that $3+k = g$ which is the genus of our given surface $S_g$. 
This is our model surface, and we let $\Sigma$ denote the model surface. 
If we set $a = A$, $b = B$, $c = C$, $d= D$, $x = X$, $y = Y$, $f(c) = Z$, then the above conditions are satisfied by construction. 

By the classification of the compact orientable surfaces, for any two pairs of disjoint non-homologous simple closed curves on the surface, 
there exists a homeomorphism which maps one pair to the other. 
(This is a special case of so-called the change of coordinates principle. See for instance \cite{FarbMargalit12}.)  
Hence, there exists a homeomorphism $\Phi$ from $\Sigma$ to $S_g$ so that $\Phi(C) = c$ and $\Phi(Z) = f(c)$. 
Now set $a = \Phi(A), b = \Phi(B), d = \Phi(D), x = \Phi(X), y = \Phi(Y)$. Then we get the desired set of curves $a, b, d, x, y$ which satisfies all the conditions together with $c, f(c)$.

For any curve $\gamma$ on $S_g$, let $T_\gamma$ be the left-handed Dehn twist about $\gamma$. 
Then by the lantern relation, we have $T_a T_b T_c T_d = T_{f(c)}T_x T_y$. 
Using the commutativity of the Dehn twists about disjoint curves, one can
rewrite the lantern relation as 
$$ T_d = T_c^{-1} T_{f(c)} T_a^{-1} T_x T_b^{-1} T_y. $$ 
Note that $T_c^{-1} T_{f(c)} = T_c^{-1} (f T_c f^{-1})= (T_c^{-1} f T_c) f^{-1}$ which is contained in the normal closure of $f$. 

As before by the change of coordinates principle, 
there exists an orientation-preserving homeomorphism $h$ of $S_g$ such that $h(c)= a$ and $h( f(c) ) = x$. 
Then $T_a^{-1} T_x = T_{h(c)}^{-1}
T_{h(f(c))} = h^{-1} T_c^{-1} T_{f(c)} h$, i.e., it is just a
conjugate of $T_c^{-1} T_{f(c)}$. Hence $T_a^{-1} T_x$ is in the normal closure of
$f$. Similarly, $T_b^{-1} T_y$ is also contained in the normal closure of
$f$. 

This shows that $T_d$ lies in the normal closure of $f$. From the fact that
there exists only one mapping class group orbit of nonseparating
simple closed curves and the Dehn twists about nonseparating simple closed curves
generate the mapping class group, we can now conclude that the entire
mapping class group $\Mod(S_g)$ is contained in the normal
closure of $f$. 
\end{proof}

Now we prove the key proposition on the sequences in the fibered cone.

\begin{prop}
\label{lem:seq} 
Let $\mathscr{C}$ be a fibered cone for a closed hyperbolic fibered 3-manifold $M$. 
Let  $\alpha \in \mathscr{C}$  and $\beta \in \overline{\mathscr{C}}$ 
be  integral classes. 
   Then there is some integer $n_0>0$ depending on $\alpha$ and $\beta$ which satisfies the following. 
 If $(S, \psi) = \alpha + n \beta \in \mathscr{C}$ is a primitive integral class for $n \ge n_0$, then 
there is an essential simple closed curve $c$ on $S$ 
  such that $c, \psi(c), \cdots, \psi^{n-1}(c)$ are disjoint, and 
  $\pm [c ] \ne [\psi(c)]$ in $H_1(S)$.   
% $[c ] \ne [\psi(c)]$ in $H_1(S)$.    
\end{prop}

\begin{proof} 
Let $n$ be a positive integer such that $\alpha + n \beta$
is a primitive integral class. 
Let $S_\alpha$ and $S_\beta$ be embedded surfaces in $M$ which represent $\alpha$ and $\beta$ respectively. 
%Note that they come with the orientations, 
Note that their orientations are assigned,  
and each connected component of those
surfaces has genus at least 2, since $M$ is a closed hyperbolic $3$-manifold. 
In what follows, we explain how to choose these representatives more explicitly.

For any primitive integral class in  $\mathscr{C}$, 
one obtains a suspension flow $\mathcal{F}$ of the monodromy. 
Fried showed that when $M$ is  a closed hyperbolic fibered 3-manifold, 
the flow $\mathcal{F}$ is invariant of $\mathscr{C}$ in the following sense: 
if one considers the suspension flows from two 
primitive integral classes in $\mathscr{C}$, then 
they are the same flow up to reparametrization and conjugation by homeomorphisms on $M$. 
Moreover Fried showed that if an embedded surface $S$ in $M$ is a fiber for a primitive integral  class in $\mathscr{C}$, then 
$S$ can be made transverse to $\mathcal{F}$, and the first return map along the flow $\mathcal{F}$
represents the monodromy  
(see \cite{fried1982geometry}, and Theorem 14.11, Lemma 14.12 in \cite{FLP}).

Surely $S_\alpha$ can be made transverse to  $\mathcal{F}$, since $\alpha \in \mathscr{C}$. 
If $\beta \in  \mathscr{C}$, then the same holds for $S_{\beta}$. 
However if  $\beta \in \partial \mathscr{C} = \overline{\mathscr{C}} \setminus  \mathscr{C}$, 
then this may or may  not be possible for representatives  of $\beta$.  
The transverse surface theorems by Mosher \cite{mosher1991surfaces} and Landry \cite{Landry19} 
including the case of compact hyperbolic $3$-manifolds tells us that, 
for any integral class $\beta \in \overline{\mathscr{C}}$, 
there exists a flow $\hat{\mathcal{F}}$ which is semi-conjugate to $\mathcal{F}$ 
so that a representative $S_{\beta}$ of $\beta$ is transverse to $\hat{\mathcal{F}}$.
Here $\hat{\mathcal{F}}$ is obtained from $\mathcal{F}$ by using the dynamic blow up of some (possibly empty) singular periodic orbits 
of  $\mathcal{F}$. 
The flow $\hat{\mathcal{F}}$ is called a dynamic blow up of $\mathcal{F}$ for $\beta \in \overline{\mathscr{C}}$.  
(The dynamic blowups of $\mathcal{F}$  may not be unique.) 
For more details of the dynamic blow up of singular orbits, 
see \cite[p.8-9]{mosher1991surfaces}, \cite[Section 3.1]{Landry19}. 
%for the dynamic blow up of singular orbits. 

We now explain some relevant properties of $\hat{\mathcal{F}}$ which are needed in the proof of Proposition \ref{lem:seq}. 
The new flow $\hat{\mathcal{F}}$ is obtained from $\mathcal{F}$ by replacing the singular orbits 
of $\mathcal{F}$ by a set of annuli such that flow lines in the interior of each annulus spiral toward boundary components of the annulus. 
Moreover 
$S_\alpha \cap \mathcal{A}$ is a union of embedded trees in $S_\alpha$, 
where $\mathcal{A}$ is the collection of annuli created during the finitely many blowups of singular orbits. 
When $\beta \in \mathscr{C}$, 
it is shown in the transverse surface theorem that $\hat{\mathcal{F}}$ is obtained from $\mathcal{F}$ along empty periodic orbits, and hence 
$\hat{\mathcal{F}}$ is the same as  $\mathcal{F}$.  
Now $S_{\beta}$ is transverse to $\hat{\mathcal{F}}$. 
From the construction of $\hat{\mathcal{F}}$, we may suppose that $S_\alpha$ is still transverse to $\hat{\mathcal{F}}$.  

For any positive integer $n$, we can consider  $n$ parallel copies of $S_\beta$, say $S_1, \ldots, S_n$ 
such that $S_i$'s are very close to each other. Whenever we are in this situation, 
the $n$ copies $S_i$'s are labeled so that for $1 \le i < n$, 
$S_i$ gets mapped to $S_{i+1}$ by the flow $\hat{\mathcal{F}}$ before
touching any other $S_j$. Note that $n$ is not fixed. 

We now describe the surgery, i.e., cut and paste on $S_\alpha, S_1, \ldots, S_n$ along the intersection locus
to get a surface $S$ which represents  $\alpha + n \beta$. 
Along each component of the intersection between $S_\alpha$
and each copy of $S_\beta$, we cut those surfaces. Locally there are
four sheets of surfaces, two from $S_\alpha$ and two from the copy of $S_\beta$. 
Glue one sheet from $S_\alpha$ to one sheet from $S_\beta$ so that the
orientations on those sheets match up. One can do the same for the
other two remaining sheets. 
The resulting surface $S$ represents  $\alpha + n \beta$. 
Clearly $S$ is  transverse to $\hat{\mathcal{F}}$. 

Note that we may assume that $S_\alpha \cap S_i$ is disjoint from
the above $\mathcal{A}$. To see this, first note that there are open disks $U$'s 
around singularities (of the unstable foliation for the pseudo-Anosov monodromy $\psi_{\alpha}$) 
on $S_\alpha$ so that $ S_\alpha \cap \mathcal{A}$ is contained in the union of the disks $U$'s. 
Now we perturb $S_i$ so that their intersection does not meet $U$'s. 
For each such disk $U$, take a disk $V$ slightly bigger than $U$ so that the closure of $U$ is contained in $V$ and $V$'s are pairwise disjoint. 
Since $S_\alpha$ is transverse to $\hat{\mathcal{F}}$, 
we may consider the small open subset of $M$ which is an I-bundle over $V$ whose fibers are segments of the flow lines of $\hat{\mathcal{F}}$. 
There is a `horizontal direction' in this I-bundle, since one can consider
the foliation of the I-bundle by the disks parallel to the disk $V$ in $S_\alpha$. 
One can consider a homotopy supported in the closure of this open I-bundle
which pushes $S_i$ along the horizontal direction so that the homotoped $S_i$ avoids the original disk $U$. 
Since the closure of $U$ is contained in $V$ and we have parallel copies of $S_i$ which are very close to each other, 
this homotopy can be applied to all of them simultaneously. 

Since $S_\alpha \cap S_i$ is disjoint from $\mathcal{A}$, the surgery does not affect the trees in $S_\alpha \cap \mathcal{A}$. 
Note that the original suspension flow $\mathcal{F}$ can be recovered from $\hat{\mathcal{F}}$ 
simply by collapsing each annulus in $\mathcal{A}$ to a closed orbit.  
%Since $S$ is fiber corresponding to $\alpha + k \beta \in \mathscr{C}$, $S$ is transverse to $\mathcal{F}$ after collapsing. 

Since $\alpha+ n \beta = [S] \in \mathscr{C}$, 
the surface $S$ can be made transverse to the original suspension flow $\mathcal{F}$. 
Now let $\hat{\Psi}$ and $\Psi$ be the first return maps on $S$ for $\hat{\mathcal{F}}$ and $\mathcal{F}$, respectively. 
Since $\hat{\Psi}$ and $\Psi$ differ only on the trees and each tree is contractible, 
$\hat{\Psi}$ and $\Psi$ are clearly homotopic to each other. Therefore
$\hat{\Psi}$ represents the monodromy  $\psi = [\Psi]$ for $\alpha+ n \beta$. 

Note that because all $S_i$ are parallel copies of $S_\beta$, any curve or
region on $S_\beta$ gives rise to a curve or region on each of the
$S_i$ that are parallel to it. 
Hence, in what follows, whenever we specify any multicurve on $S_\beta$ we implicitly specify multicurves on all of the $S_i$ which are parallel to each other.

Let $C$ be a multicurve on $S_\beta$, 
such that all the connected components of  $S_\beta \backslash C$ have genus $0$ with three ends (Figure \ref{fig_3valent_2}(1)). 
Furthermore, we assume that every intersecting curve between $S_\alpha$ and $S_\beta$ is parallel to one of the curves in $C$. 
Such a multicurve $C$ always exists. 
To construct one, group the intersecting curves between $S_\alpha$ and $S_\beta$ into parallel families, 
choose one in each parallel family and use them to form a multicurve $C'$. 
Now, if some connected component of  $S_\beta\backslash C'$ has genus greater than $0$, or has more than three ends, 
then we can add an extra curve to $C'$ to break it into components of lower complexity, 
and repeat this process until all the connected components of $S_{\beta} \backslash C'$ have genus $0$ with three ends.

\begin{center}
\begin{figure}[t]
\includegraphics[height=3.6cm]{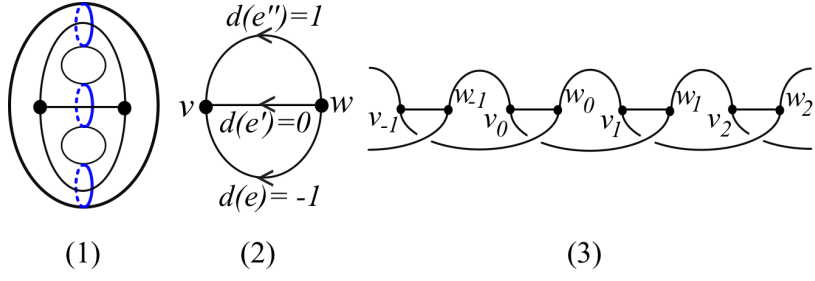}
\caption{(1) A multicurve $C$ together with its 3-regular graph $G$ 
on $S_ {\beta} \simeq $ closed surface of genus $2$. 
(2) An example of a cochain $d$ on $G$: 
for three edges from $w$ to $v$, their values are $-1$, $0$, $1$ respectively. 
(3) ${\Bbb Z}$-fold cover $G'$ corresponding to $d$ of (2).}  
\label{fig_3valent_2}
\end{figure}
\end{center}

Now we make use of the graph theoretic lemma below.

\begin{lem}
\label{lem_3regular}
Let $G$ be a 3-regular finite graph. Let $d$ be an integer valued cellular cochain on $G$ whose value on each edge is bounded above by $k \ge 0$, 
and let $G'$ be the $\mathbb{Z}$-fold cover constructed from $d$ (i.e. the vertices of $G'$ are $\mathbb{Z}$-copies of the vertices of $G$ and each edge $e$ in $G$ 
from $w$ to $v$ is lifted to edges from the $j$th lift of $w$ to the $(j+d(e))$th lift of $v$, 
see Figures \ref{fig_3valent_2}(2) and \ref{fig_3valent_2}(3)). 
Then there is some $R$ depending only on $k$ and the number of edges $|E(G)|$ of $G$ such that 
$G'$ has a simple loop $\gamma'$ of length no more than $2R$.
\end{lem}

\begin{proof}
Suppose there are no such loops of length less than $2R$ in $G'$ for any $R$. 
Then the $R$-neighborhood
(i.e., neighborhood with radius $R$ assigning each edge length 1)
of any vertex $v_0$ in $G'$ must be a 3-valence tree.
Hence it contains $3\times (2^R-1)$ edges. 
However, such a neighborhood must contain at most $(2Rk+1)|E(G)|$ edges. 
(This is because in $R$ steps, one can travel up at most $Rk$ levels, i.e., 
$Rk$ copies of the fundamental domain, 
or travel down at most $Rk$ levels. 
Together with the original level, there are $(2Rk+1)$ levels in total that one might be able to pass through, 
and hence there are at most  $(2Rk+1)|E(G)|$ edges in them.) 

Since the exponential function grows faster than the linear one, one can set $R$ sufficiently large to reach a contradiction.
\end{proof}

\begin{center}
\begin{figure}[t]
\includegraphics[height=2.7cm]{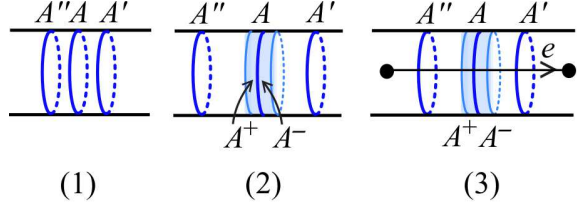}
\caption{(1) Parallel curves on $S_{\beta}$ which are some components of the intersection between $S_{\alpha}$ and $S_{\beta}$. 
(2) Annular neighborhood of $A$ and the side of $A^{\pm}$. 
(3) For an edge $e$ starting from the side of  $A^+$, $A$ contributes to $d(e)$ by $+1$.}   
\label{fig_cochain}
\end{figure}
\end{center}

We continue the proof of Proposition \ref{lem:seq}.
Note that the multicurve $C$ above gives a pants decomposition of $S_\beta$.
Let $G$ be the 3-regular graph where each vertex corresponds to a pair of pants in the pants decomposition of $S_{\beta}$, 
and each edge corresponds to the component of the multicurve between two pairs of pants. 
(See Figure \ref{fig_3valent_2}(1).)
Now we define the cochain $d$ on $G$ which only depends on $S_{\alpha}$ and $S_{\beta}$ 
as follows. (See Figure \ref{fig_cochain}.)

Consider the surface $S$ obtained from the cut and paste construction of $S_ {\alpha}$ and $n$ copies of $S_{\beta}$. 
If a curve $A$ is one component of the intersection between 
$S_\alpha$ and $S_{\beta}$, we cut $S_\beta$ along $A$ (hence we cut each copy of $S_{\beta}$ along a curve corresponding to $A$) 
which results in two boundary curves for each copy of $S_{\beta}$, say $A^+$ and $A^-$. The labeling $A^+$ and $A^-$ are determined as follows: 
in the surface obtained from $S_\alpha$ and the copies of $S_\beta$ via the cut and paste construction, 
an annular piece of $S_\alpha$ connecting the $i$th copy of $S_\beta$ to the $(i+1)$th copy of $S_\beta$ is attached to 
the $i$th copy of $S_\beta$ along $A^+$ (the index of each copy of
$S_{\beta}$ is understood as an integer modulo $n$). We label the
other boundary componet $A^-$.  

Now the labeling on each copy of $S_\beta$ is well-defined, and if one
considers an annular neighborhood of $A$, 
then one can make sense of that one side is the side of $A^+$ 
and the other side is the side of $A^-$. 

Let us consider an edge $e$ on $G$ which intersects the curve $A$.  
If $e$ is with the orientation so that it starts from the side of $A^+$ and go over the side of $A^-$, 
then $A$ contributes to $d(e)$ by $+1$, and $A$ contributes to $d(e^{-1})$ by $-1$, 
where $e^{-1}$ is the same edge as $e$ with the opposite orientation. 
The number $d(e)$ is obtained by summing up all the contributions of curves in $S_\alpha\cap S_\beta$ that the edge $e$ passes through. 
Note that the cochain $d$ does not depend on $n$ but only on $S_\alpha$ and $S_\beta$, 
since we consider copies of $S_{\beta}$ very close to each other, the intersection with 
$S_\alpha$ looks exactly the same in any copy of $S_{\beta}$.

Let $k$ be the maximum of the values of $d$ on all edges on $G$, 
and let $R$ be the constant from Lemma \ref{lem_3regular}. Now let $n$ be any integer so that  $n \geq 2Rk+2$, and 
consider the surface $S$ obtained from $S_\alpha$ and $n$ copies of $S_\beta$ by a cut and paste construction. 
(In other words, here we will argue that the integer $n_0$ in  Proposition \ref{lem:seq} can be chosen as $2Rk+2$.) 
Let $\gamma'$ be a simple loop in $G'$ in Lemma \ref{lem_3regular}. 
The fact that $|d(e)|\leq k$ implies that $\gamma'$ passes through 
at most $2Rk+1$ consecutive fundamental domains of the deck group action on $G'$. 
The embedding of these $2Rk+1$ fundamental domains, together with one more, to $2Rk+2$ copies of $S_\beta$ 
after the surgery, sends $\gamma'$ to some simple loop $\gamma$ on the surface $S$. 

Let $c \in C$ be a component of the multicurve on $S_\beta$ and let $c_i$ be the corresponding copies of $c$ 
on the $i$th copy $S_i$ of $S_\beta$. 
Suppose that $c$ is chosen such that 
$c_l$ is crossed by $\gamma$ once for some $l$,
and that $\gamma$ does not cross the lowest copy $S_1$ (see Figure \ref{figure:nonhomologous}). 
One can choose such $c$, since the length of $\gamma'$ is no more than $2R$. 
Note that all $c_i$ survives under surgery because they do not cross 
the intersections between $S_i$ and $S_\alpha$. Furthermore, except for the top $c_n$, their images under the first return map are
$\psi(c_i)=c_{i+1}$. 
By construction of $S$, 
it follows that 
$c_1, \psi(c_1)= c_2, \cdots, \psi^{n-1}(c_1)= c_n$ are disjoint. 
For the proof of Proposition \ref{lem:seq}, 
we only need to show that $[c_2]\pm [c_1]$ is not homologous to $0$. 
(This also implies that $c_1$ on $S$ is essential.)
To do so, one only needs to show that 
$$(\psi_*^{l-2}+\psi_*^{l-3}+\dots + \mathrm{id}_*)([c_2]-[c_1])=[c_l]-[c_1]$$
and
\[(\psi_*^{l-2}-\psi_*^{l-3}+\dots +(-1)^{l-2}\mathrm{id}_*)([c_2]+[c_1])=[c_l]+(-1)^{l-2}[c_1]\]
are not $0$.
Since $\gamma$ passes through $c_l$ and it does not pass through $c_1$, 
simple closed curves $c_l$ and $c_1$ do not bound a subsurface.
Therefore $[c_l]\not=\pm[c_1]$.
This completes the proof of Proposition \ref{lem:seq}. 
\end{proof}

\begin{figure} 
  \begin{tikzpicture}
    \draw[dotted] (0,5)--(2,5);
    \draw[dotted](2,0)--(0,0);
    \draw[dotted] (0,0.5)--(2,0.5);
    \draw[dotted] (0,1)--(2,1);
    \draw[->] (0.5,3)--(0.5,4.8)--(1.5,4.8)--(1.5, 0.4)--(0.5, 0.4)--(0.5,3);
    \node at (2.3,0.1) {$S_1$};
    \node at (2.3,3.5) {$S_l$};
    \draw[-, very thick] (0.4,0)--(0.6,0);
    \draw[-, very thick] (0.4, 3.5)--(0.6,3.5);
    \draw[dotted](0, 3.5)--(2, 3.5);
    \node at (0.5, -0.2) {$c_1$};
    \node at (0.3, 3.5) {$c_l$};
    \node at (1.7, 2.5) {$\gamma$};
  \end{tikzpicture}
  \caption{The horizontal line segments (with dots) represent the copies $S_1, S_2, \cdots$ of $S_\beta$, 
  and the curve with arrow represents the loop $\gamma$ which passes through $S_l$ but not the lowest copy $S_1$ of $S_{\beta}$.}
  \label{figure:nonhomologous}
 \end{figure}
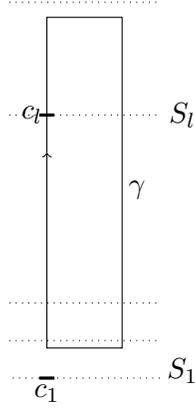

We now consider a compact hyperbolic fibered 3-manifold $M$. 
In order to obtain an estimate for the asymptotic translation length of monodromies from the arithmetic sequences in the fibered cone for $M$, 
we show the following variant of Proposition \ref{lem:seq}. 

\begin{prop}
\label{lem:seq2} 
Let $\mathscr{C}$ be a fibered cone  for a compact hyperbolic fibered 3-manifold $M$ possibly with boundary.  
Let  $\alpha \in \mathscr{C}$  and $\beta \in \overline{\mathscr{C}}$ 
be  integral classes. 
Suppose $(S, \psi) = \alpha + n\beta \in  \mathscr{C}$ is a primitive integral class for an integer $n\ge 2$. 
  Then there is an essential simple closed curve $c$ on $S$ or essential arc on $S$ 
  %(when $\partial M \ne \emptyset$)
  so that  $c, \psi(c), \cdots, \psi^{n-1}(c)$ are disjoint.  
  In particular we have 
  $$\ell_\mathcal{\mathcal{C}}(\psi) \leq \dfrac{2}{n-1}.$$
\end{prop} 

\begin{proof} 
Let $\mathcal{F}$ be the suspension flow for the fibered cone $ \mathscr{C}$. 
In \cite[Appendix A]{Landry19}, Landry generalized Fried's theory on the fibered cone (for closed hyperbolic fibered 3-manifolds) 
to the case of compact  hyperbolic fibered 3-manifolds $M$ possibly with boundary. 
In particular $\mathcal{F}$ is invariant of $ \mathscr{C}$ as well. 
Then we use the transverse surface theorem \cite{mosher1991surfaces,Landry19} for  compact  hyperbolic fibered 3-manifolds $M$ again. 
Let 
%$\mathcal{F}$ be the suspension flow for the fibered cone $ \mathscr{C}$, and let 
$\hat{\mathcal{F}}$ be a dynamic blow up of $\mathcal{F}$ for
$\beta \in \overline{\mathscr{C}}$. 
We can take representatives $S_\alpha$ and $S_\beta$ of $\alpha$ and $\beta$ respectively 
so that $S_{\alpha}$ and $S_{\beta}$ are transverse and 
they intersect the new flow $\hat{\mathcal{F}}$ transversely. 
We may assume that 
$S_{\alpha}$ and $S_{\beta}$ intersect minimally, i.e., 
the number of components of the intersection between $S_ {\alpha}$ and $S_{\beta}$ is minimal 
among all representatives of $\alpha$ and $\beta$. 
The surface obtained from $S_{\alpha}$ and $S_{\beta}$ by a cut and paste construction is a fiber of the fibration associated with $\alpha+ \beta \in \mathscr{C}$. 
This implies that $S_{\alpha}$ and $S_{\beta}$ are minimal representatives of $\alpha$ and $\beta$. 
%and $S_{\alpha}$ and $S_{\beta}$ intersect minimally, i.e, 
%the number of components between $S_ {\alpha}$ and $S_{\beta}$ is minimal 
%among representatives of $\alpha$ and $\beta$. 
Do surgery at the intersection locus of $S_\alpha$ and $n$ copies of $S_\beta$ to obtain a surface $S$ representing $\alpha+n\beta$. 
We now find the desired essential simple closed curve on $S$ or an essential arc $c$ on $S$. 
%(when $\partial M \ne \emptyset$). 
Let $c$ be one of the intersection curves or arcs 
between $S_\alpha$ and $S_{\beta}$, and 
let $S_1$ be the the lowest copy of $S_{\beta}$.  
The fact that $c$ is essential on $S_\alpha$ and on $S_\beta$ follows 
from the fact that the intersection between $S_\alpha$ and $S_{\beta}$ is minimal 
(see \cite{Thurston86} or \cite[Lemma 5.8]{calegari2007foliations}).  
It is not hard to see from the cut and past construction that $c$ is also essential on $S$. 
%To see $c$ is essential on $S$ also, suppose otherwise, i.e. that $c$ bounds a disk, an annuli or a bigon (the other side being part of the boundary in the latter two cases). If the disk, annuli or bigon is from either $S_\alpha$ or $S_\beta$, then the surfaces $S_{\alpha}$ and $S_{\beta}$ are not intersecting minimally, which is a contradiction. If the disk (or annuli, or bigon) is not from $S_{\alpha}$ or $S_{\beta}$, since it contains some part of $S_\alpha$ and some part of $S_\beta$, it must contain some other circle or arc arising from the intersection of $S_\alpha$ and $S_\beta$. It is easy to see that any such arc or circle must bound a disc, annuli or bigon themselves. We say one circle, or arc, to be ``inside'' another if the second circle or arc bounds a disc, annuli or bigon which contains the first. Now consider the inner most circle or arc, then one can see that the surfaces don't intersect minimally. This is a contradiction again. The argument for when it bounds an annuli or bigon are similar.

From the choice of $c$, it follows that $c$ and $\psi^{n-1}(c)$ are disjoint. 
They are distinct in the arc and curve complex $\mathcal{AC}(S)$, since $\psi$ is pseudo-Anosov. 
Thus the distance between $c$ and $\psi^{n-1}(c)$ in $\mathcal{AC}(S)$ equals $1$.  
This implies that $ (n-1)\ell_\mathcal{AC}(\psi) =  \ell_\mathcal{AC}(\psi^{n-1}) \leq 1$ (cf.  \cite[Lemma 2.1]{KinShin18}), 
where $\ell_\mathcal{AC}(\psi)$ is the asymptotic translation length of $\psi$ on $\mathcal{AC}(S)$. 
It is known  that the inclusion map $\C(S) \ra \AC(S)$ is 2-bilipschitz (see, for instance, \cite[Lemma 2.2]{MasurMinsky00} or \cite{KorkmazPapadopoulos10}). 
In particular, this tells us that 
$$\ell_\mathcal{\mathcal{C}}(\psi) \le 2 \ell_\mathcal{\mathcal{AC}}(\psi).$$
Thus we have 
$\ell_\mathcal{\mathcal{C}}(\psi) \le 2 \ell_\mathcal{\mathcal{AC}}(\psi) \le \frac{2}{n-1}$. 
This completes the proof. 
%and hence $\ell_\mathcal{\mathcal{C}}(\psi)\lesssim \ell_\mathcal{\mathcal{AC}}(\psi) \leq \dfrac{1}{n-1}$, the first $\lesssim$ sign is because the curve complex and the arc and curve complex are quasi-isometric.
\end{proof}

\begin{remark}
\label{rem_BSW}
In \cite{BaikShinWu18}, Theorem \ref{thm:BaikShinWu} was proved 
in the case of closed hyperbolic fibered 3-manifolds. 
We note that almost the same proof can be adopted to the case of compact hyperbolic fibered 3-manifold. 
In fact, one only needs to modify the last paragraph (after Lemma 8) in the proof of Theorem 5 in  \cite{BaikShinWu18}
to allow $\gamma$ and $\gamma'$ to be either an essential simple closed curve or an essential simple arc. 
Then one obtains the same conclusion of Theorem \ref{thm:BaikShinWu} 
by the fact that  inclusion map $\C(S) \ra \AC(S)$ is 2-bilipschitz as 
in the proof of Proposition \ref{lem:seq2} in this paper. 
\end{remark}

\medskip
%%%%%%%%%%%%%%%%%%%%%%%%%%%%%%%%%%%%%%%%%%%%%%%%
%
%						Applications 
%
%%%%%%%%%%%%%%%%%%%%%%%%%%%%%%%%%%%%%%%%%%%%%%%%
\section{Applications of arithmetic sequences} 
%%%%%%%%%%%%%%%%%%%%%%%%%%%%%%%%%%%%%%%%%%%%%%%%
%
%						Asympttic translation lengths in fibered cones
%
%%%%%%%%%%%%%%%%%%%%%%%%%%%%%%%%%%%%%%%%%%%%%%%%
\medskip
\subsection{Asymptotic translation lengths in fibered cones}	\label{section:asymptotictranslationlengths}

In this section, we show the following estimate for the asymptotic translation lengths in the curve complexes.

\begin{thm} \label{thm:newbound} 
Let $F$ be a fibered face of a compact hyperbolic fibered 3-manifold possibly with boundary.
Then there exists a constant $C$ depending on $F$ such that 
for any primitive integral class $(S, \psi) \in \mathscr{C}_F$, 
we have
$$\ell_{\C} (\psi) \leq \frac{C}{|\chi(S)|}.$$
\end{thm}

To prove this theorem, we need the following lemma about rational cones. Here a rational cone in Euclidean space $\mathbb{R}^m$ 
is the set of the points of the form 
$$\{ {\bf x}= (x_1, \ldots, x_m) \in \mathbb{R}^m : A {\bf x}^{t} \geq \textbf{0}\}$$ 
for some $k \times m$ matrix $A$ with integer entries. 
(${\bf x}^{t}$ is the transpose of ${\bf x}$.) 
%\marginal{Isn't ${\bf x}^t$ the standard notation for the transpose? (S)}
We further assume that this set has nonempty interior.  

\begin{lem}\label{lem:cone}
Let $P$ be a rational cone in $\mathbb{R}^m$, and let $int(P)$ be its interior. 
Then there exist two nonempty finite sets $\Omega_0 \subset int(P) \cap\mathbb{Z}^m$ and $\Omega \subset P\cap \mathbb{Z}^m$
so that
 $$int(P) \cap\mathbb{Z}^m=\{a+\sum_{b\in \Omega} k_bb: a\in\Omega_0, k_b\in\mathbb{Z}, k_b\ge 0\}.$$
\end{lem}

\begin{proof}
It is a classical result (cf. \cite[Proposition 3.4]{thurston2014entropy}) that  
$P\cap\mathbb{Z}^m$ is a finitely generated monoid.  
Let $\Omega$ be a finite set of generators of $P\cap\mathbb{Z}^m$, and 
let   
\[\Omega_0=\{\sum_{b\in W}b : W\subset \Omega, W\not\subset F \text{ for all faces }F\text{ of }\partial P\}.\]
Here a face of $\partial P$ is a polytope of dimension $m-1$ which is the intersection of $\partial P$ with a $m-1$ dimensional subspace of $\mathbb{R}^m$. 
Note that $W$ can possibly contain only a single point in $int(P)$. 
Clearly $\Omega_0$ is a finite set with at most $2^{|\Omega|}$ elements.

Note that 
a linear combination of elements in $\Omega$ with non negative coefficients 
lie on a face of $\partial P$ if and only if  all the coefficients for those generators that are not on this face are $0$. In other words, if $\sum_{b\in\Omega}k_b b$ is in $int(P)$ and $k_b$ are all non negative, then the set $\{b\in \Omega:k_b\ge 1\}$ must not be contained in any face of $\partial P$. Hence 
\[int(P) \cap\mathbb{Z}^m=\{a+\sum_{b\in \Omega} k_bb: a\in\Omega_0, k_b\in\mathbb{Z}, k_b\ge 0\}\]
and in particular $\Omega_0\subset int(P)\cap\mathbb{Z}^m$ as we desire. 
\end{proof}

%$$P= \Bigl\{ {\bf x}= (x_1, x_2) \in \mathbb{R}^2 : 
%\left(\begin{array}{cc}0 & 1 \\3 & -2\end{array}\right) \left(\begin{array}{c}x_1 \\x_2\end{array}\right) 
%>  \left(\begin{array}{c}0 \\0\end{array}\right)
%\Bigr\}.$$

Here is an example of the two finite sets $\Omega_0$ and $\Omega$ for a rational cone in ${\Bbb R}^2$. 
%illustrating the proof of Lemma \ref{lem:cone}. 

\begin{ex}
Let us consider the following rational cone in $\mathbb{R}^2$. 
$$P= \Bigl\{ {\bf x}= (x_1, x_2) \in \mathbb{R}^2 : 
\left(\begin{array}{cc}0 & 1 \\3 & -2\end{array}\right) \left(\begin{array}{c}x_1 \\x_2\end{array}\right) 
\ge  \left(\begin{array}{c}0 \\0\end{array}\right)
\Bigr\}.$$
One can take $\Omega = \{b_1=(1,0), b_2=(1,1), b_3= (2,3)\}$  
as  a set of generators of $P \cap \mathbb{Z}^2$. 
There are two faces of  $\partial P$. 
One is $\{(x, 0): x \ge 0\}$ which contains $\{b_1\}$ as a subset, and the other is 
$\{(x, \frac{3x}{2} ): x \ge 0\}$ which contains $\{b_3\}$ as a subset. 
One sees that $\Omega_0$ consists of five elements, 
$b_2$, $b_1+ b_2=(2,1)$, $b_1+ b_3= (3,3)$, $b_2+b_3=(3,4)$ and 
$b_1+ b_2+b_3=(4,4)$. 
\end{ex}

\begin{proof}[Proof of Theorem \ref{thm:newbound}] 
For a fibered cone $\mathscr{C}$, 
the closure $\overline{\mathscr{C}}$  is a rational cone in $H^1(M)$, 
because the unit Thurston norm ball is a polytope whose vertices are rational points \cite{Thurston86}. 
By Lemma \ref{lem:cone}, if an integral class $\delta$ is in $\mathscr{C}$, 
then it can always be written of the form $\delta=a+\sum_{b\in\Omega}k_b b$, 
where $a\in\Omega_0$ and $k_b$ is a non negative integer. 
If $S$ is a norm-minimizing surface of $\delta$, then we have 
$\|\delta\|= |\chi(S)|$ and it is bounded above by 
$$\max(1, \max_{b \in \Omega}(k_b))(\|a\|+\sum_{b\in\Omega}\|b\|).$$
Hence, when $|\chi(S)|>\displaystyle\max_{a\in\Omega_0} \|a\|+\sum_{b\in\Omega} \|b\|$, we have 
 $$|\chi(S)| \le  \max_{b \in \Omega}(k_b)(\|a\|+\sum_{b\in\Omega}\|b\|).$$
 Therefore 
$$\max_{b \in \Omega}(k_b)\geq \frac{|\chi(S)|} {\|a\|+\sum_{b\in\Omega} \|b\|} \geq
\frac{|\chi(S)|} {\max_{a\in\Omega_0} \|a\|+\sum_{b\in\Omega} \|b\|}.$$
%$$\max_{b \in \Omega}(k_b)\geq {|\chi(S)|\over \|a\|+\sum_{b\in\Omega} \|b\|} \geq
%{|\chi(S)|\over \max_{a\in\Omega_0} \|a\|+\sum_{b\in\Omega} \|b\|}.$$
%Hence $\max(k_b)\geq {|\chi(S)|\over max_{a\in\Omega_0} \|a\|+\sum_{b\in\Omega} \|b\|}$. 
%Let $b_m=\arg\max_{b\in\Omega}k_b$ (i.e, $b_m$ is the $b$ in $\Omega$ that maximizes $k_b$), 
Let  $b_m$ be the $b$ in $\Omega$ that maximizes $k_b$. 
We set 
$\alpha=a+\sum_{b\in\Omega, b\not=b_m}k_bb$, $\beta=b_m$ and $n=k_{b_m}$. 
We have $\alpha \in \mathscr{C}$ and $\beta \in \overline{\mathscr{C}}$. 
Then $\delta$ is written by 
$\delta=\alpha+n \beta$ with 
$$n\geq \frac{|\chi(S)|} {\max_{a\in\Omega_0} \|a\|+\sum_{b\in\Omega} \|b\|}.$$  
%$n\geq {|\chi(S)|\over \max_{a\in\Omega_0} \|a\|+\sum_{b\in\Omega} \|b\|}$. 
Note that the denominator in the right hand side only depends on the fibered cone. 
Now the theorem follows from Proposition \ref{lem:seq2}, 
since the set of primitive integral classes $\delta$ with 
$\|\delta\| \le \displaystyle\max_{a\in\Omega_0} \|a\|+\sum_{b\in\Omega} \|b\|$ is finite. 
\end{proof}

%%%%%%%%%%%%%%%%%%%%%%%%%%%%%%%%%%%%%%%%%%%%%%%%
%
%							Normal generation in the fibered cone
%
%%%%%%%%%%%%%%%%%%%%%%%%%%%%%%%%%%%%%%%%%%%%%%%%
\medskip
\subsection{Normal generation in the fibered cone}	 \label{section:normalgeneration}

In this section, we prove the following theorem as a partial answer to
Conjecture \ref{conj:normalgeneration}. 

\begin{thm} \label{thm:normalgeneration} 
Let $F$ be a fibered face of a closed  hyperbolic fibered 3-manifold $M$, 
and let $L$ be a 2-dimensional rational subspace of $H^1(M)$.  
Then for all but finitely many primitive integral classes $(S, \psi)$ in $\mathscr{C}_F \cap L$, $\psi$ normally generates $\Mod(S)$. 
In particular, if the rank of $H^1(M)$ equals $2$, then 
Conjecture \ref{conj:normalgeneration} is true.
\end{thm}

\begin{center}
\begin{figure}[t]
\includegraphics[height=3.8cm]{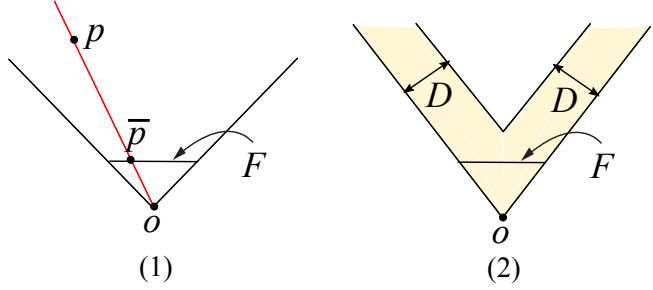}
\caption{(1) Fibered face $F$ in the fibered cone $\mathscr{C}$. 
($p$ and $\overline{p}$ lie on the same ray in $\mathscr{C}_F$ passing through the origin.) 
(2) Subset $\mathcal{N}_D \subset \mathscr{C}$.}  
\label{fig_cone}
\end{figure}
\end{center}

For the proof of Theorem \ref{thm:normalgeneration}, 
we first prove the following result.

\begin{thm} \label{thm:awayfromnbhdofboundary} 
Let $\mathscr{C}$ be a fibered cone of a closed hyperbolic fibered 3-manifold $M$. 
Then there exists some $x \in \mathscr{C}$ such that for each 
primitive integral class $(S, \psi) \in x+ \mathscr{C}$, 
$\psi$ normally generates $\Mod(S)$, 
where  $x+\mathscr{C}=\{x+v: v\in\mathscr{C}\}$.
\end{thm}

\begin{proof}
Let $d$ be any Euclidean metric on $H^1(M)$. 
Let $F$ be the fibered face corresponding to $\mathscr{C}$. 
 For every point $p\in\mathscr{C}$, let $\overline{p}$ be the intersection of $F$ 
  with the ray starting from the origin and passing $p$ (Figure \ref{fig_cone}(1)). 
By \cite[Corollary 5.4]{McMullen00}, 
  we have a real analytic, strictly concave and degree-1 homogeneous function 
  $y= 1/{\log K(\cdot)}$ defined on $\mathscr{C}$, such that the stretch factor $\lambda(p)$ for 
  $p\in\mathscr{C}$ is equal to  $K(p)$ 
  and $y(p) = 1/{\log K(p)} \to 0$ as $p \to \partial F$. 
%  In particular $y$ is a continuous function on the closure $\overline{F}$ of $F$. 
  The concavity implies that there must be some $k>0$ (independent on the choice of $\overline{p}$) 
  so that 
  \[{1\over \log(K(\overline{p}))}\geq k \cdot d(\overline{p}, \partial \mathscr{C}).\]
   A way to see the existence of $k$ is as follows: concavity of $y$ implies that there is some point $p_0 \in F$, 
   where $y(p_0)>0$. 
   Then, for any point $\overline{p} \in F$, consider the line segment from $p_0$ to the boundary of $F$ passing through $\overline{p}$. 
   Then concavity of $y$ means that on this line segment, 
   $y$ is bounded from below by the linear function $L$ which takes value $0$ at one end and $y(p_0)$ at another end. 
   Hence it has a slope $s= s(\overline{p})$ that depends on $\overline{p}$ and 
   $s= s(\overline{p})$ is continuous on $\overline{p}$. 
   On the other hand, the function $d(\cdot, \partial\mathscr{C})$, restricted to this line segment, is piecewise linear, 
   and hence it is also bounded from above by a linear function $L'$
   taking value $0$ at the end on $\partial F$. 
   We choose such linear function $L'$ with the smallest slope $s'= s'(\overline{p})$. 
   Then $s'= s'(\overline{p})$  is continuous on $\overline{p}$. 
   Now $k$ can be chosen as any number below the ratio $s/s'$ 
   between these two slopes. 
   As both slopes depends continuously on $\overline{p}$, and $F$ has compact closure, we can choose a universal $k$ that works on the whole face $F$.

  Furthermore,  the degree-$1$ homogeneity implies that
  \[{1\over \log(K(p))}={d(0,p)\over d(0,\overline{p})}\cdot {1\over \log(K(\overline{p}))}\] 
  
For $D>0$, we  consider the following set $\mathcal{N}_D$ (Figure \ref{fig_cone}(2)). 
  $$\mathcal{N}_D=\{p \in \mathscr{C}: d(p, \partial\mathscr{C})\leq D\}.$$ 
    From the above computation, the stretch factor for 
  $p\in \mathscr{C}\backslash \mathcal{N}_D$ satisfies 
  $$\lambda(p) = 
  e^{\log K(p)}=(e^{\log K(\overline{p})})^{d(0,\overline{p})\over d(0,p)} \le (e^{1\over k d(\overline{p},\partial\mathscr{C})})^{d(0,\overline{p})\over d(0,p)}=e^{1\over k d(p,\partial\mathscr{C})}\le e^{\frac{1}{kD}}.$$
   %$\leq (e^{\frac{h}{kD}})^{\frac{1}{h}}=e^{\frac{1}{kD}}$. 
  Hence as long as $D$ is sufficiently large, $\lambda(p)$ can be made to be as close to 1 as needed. 
  In particular it is smaller than $\sqrt{2}$ when $D$ is large enough.  
  This together with \cite[Theorem 1.2]{LanierMargalit18} shows that for some $D$, all primitive integral classes in $\mathscr{C}\backslash\mathcal{N}_D$ are normal generators. The theorem now follows by picking an arbitrary $x\in\mathscr{C}\backslash\mathcal{N}_D$, 
  due to the fact that the boundary of $\mathcal{N}_D$ must be parallel to that of $\partial \mathscr{C}$ itself 
  (See Figure \ref{fig_cone}(2)).
\end{proof}

The next result follows immediately from Lemma \ref{lem:normalcriterion} and Proposition \ref{lem:seq}.

\begin{thm} 
\label{thm:boundaryparallel sequences} 
Let $\mathscr{C}$ be a fibered cone of a closed hyperbolic fibered 3-manifold. 
Suppose that $(S_{\alpha_n}, \psi_{\alpha_n})$ is a sequence of primitive integral classes in 
$\mathscr{C}$ such that 
$\alpha_n = v + nw$, 
where $v \in \mathscr{C}$ and $w \in \overline{\mathscr{C}} $ are fixed integral classes. 
Then $\psi_{\alpha_n}$ normally generates $\Mod(S_{\alpha_n})$ 
for sufficiently large $n$. 
\end{thm}
%Let $\mathscr{C}$ be a fibered cone of a hyperbolic fibered 3-manifold. 
%Suppose $(S_i, \psi_i)$ is a sequence in 
%$\mathscr{C}$ parallel to the boundary, i.e., it is of the form $(v +iu)_{i \in \mathbb{N}}$ 
%where $v$ is an integral class in the fibered cone, and $u$ is an integral 
%class in the closure of the fibered cone. 
%Then $\psi_i$ normally generates $\Mod(S_i)$ 
%for sufficiently large $i$. 
%{\bf (Do you need to assume $u$ and $v$ are primitive? 
%I guess only assumption you need is $v+ iu$ are primitive.  
%I wrote this statement below. (E)) No we don't. (C)}

%
%\begin{proof}
%This follows immediately from Lemma \ref{lem:normalcriterion} and Proposition \ref{lem:seq}. 
%\end{proof}

%\noindent
%{\bf Theorem \ref{thm:boundaryparallel sequences}.} 
%{\bf Can I write your thm 3.5 in this way? let me know if something wrong.(E) I think this statement is good. (C)} 
%Let $\mathscr{C}$ be a fibered cone of a hyperbolic fibered 3-manifold. 
%Suppose that $(S_{\alpha_n}, \psi_{\alpha_n})$ is a sequence of primitive integral classes in 
%$\mathscr{C}$ such that 
%$\alpha_n = v + nw$, 
%where $v \in \mathscr{C}$ and $w \in \overline{\mathscr{C}} $ are integral classes. 
%Then $\psi_{\alpha_n}$ normally generates $\Mod(S_{\alpha_n})$ 
%for sufficiently large $n$. 
%\medskip

% \begin{proof}
%This follows immediately from Lemma \ref{lem:normalcriterion} and Proposition \ref{lem:seq}.
%\end{proof}

We are now ready to prove Theorem \ref{thm:normalgeneration}. 

%\begin{proof}[Proof of Theorem \ref{thm:normalgeneration}]
%(Chenxi)
%Theorem \ref{thm:awayfromnbhdofboundary} shows that there is a uniform neighborhood $\mathcal{N}$ of $\partial\mathscr{C}\cap L$ so that all primitive integer points in $\mathscr{C}\cap L\backslash \mathcal{N}$ normally generate the corresponding mapping class group. Because $L$ is of dimension $2$, the integer points in $\mathscr{C}\cap L\cap \mathcal{N}$ is the union of finitely many sequences of the form $(v+iu)_{i\in\mathbb{N}}$, hence by Theorem 3.5 all but finitely many primitive ones among them are normal generators.
%\end{proof}

\noindent
\begin{proof}[Proof of Theorem \ref{thm:normalgeneration}]
Let $L$ be a 2-dimensional rational subspace of $H^1(M)$ satisfying the assumption of Theorem \ref{thm:normalgeneration}. 
Theorem \ref{thm:awayfromnbhdofboundary}  says that there is some $x\in\mathscr{C}$
so that all primitive integral classes $(S, \psi)$ in $x+\mathscr{C}$ normally generate $\mathrm{Mod}(S)$. 
In particular this holds for all primitive integral classes in $(x+\mathscr{C})\cap L$. 
Because $L$ is of dimension $2$, the integral classes in $(\mathscr{C}\backslash (x+\mathscr{C}))\cap L$ are the union of finitely many sequences 
of the form $(v+nw)_{n\in\mathbb{N}}$, where 
$v \in \mathscr{C}$ and $w \in \overline{\mathscr{C}} $. 
Thus by Theorem \ref{thm:boundaryparallel sequences},  
for all but finitely many primitive integral classes $(S, \psi)$ in $(\mathscr{C}\backslash (x+\mathscr{C}))\cap L$, 
$\psi$ normally generates $\mathrm{Mod}(S)$. 
This completes the proof.
\end{proof}

\begin{remark}
\label{remark:dim}
Our approach to Theorem \ref{thm:normalgeneration} 
does not work when the dimension of the rational subspace $L$ of $H^1(M)$ is more than $2$. 
This is because in this case, 
the intersection $(\mathscr{C}\backslash (x+\mathscr{C}))\cap L$ no longer consists of finitely many sequences 
of primitive integral classes of the form $v + nw$, where 
$v \in \mathscr{C}$ and $w \in \overline{\mathscr{C}} $. 
\end{remark}

%%%%%%%%%%%%%%%%%%%%%%%%%%%%%%%%%%%%%%%%%%%%%%%%
%
%					Sequences in the fibered cone of the magic 3-manifold
%
%%%%%%%%%%%%%%%%%%%%%%%%%%%%%%%%%%%%%%%%%%%%%%%%
\medskip

\section{Sequences in the fibered cone of the magic manifold} \label{section:magicmanifold}

Let $\mathcal{C}_3$ be the $3$ chain link  in $S^3$ as in Figure~\ref{fig_3chainmagic}(1). 
The magic manifold $N$ is the exterior of $\mathcal{C}_3$ 
(hence $\partial N$ consists of three boundary tori), and 
it is a hyperbolic and fibered $3$-manifold. 
We  give some background on invariant train tracks in Section \ref{sec:traintrack}
and we discuss the fibered cone of $N$ in Section \ref{sec:magicmanifold}.
We compute the upper and lower bounds of the asymptotic translation length of
particular sequences  in the fibered cone of $N$ in Sections \ref{sec:lowerbound} and \ref{sec:upperbound}.
Then we  prove Theorem \ref{thm:allrational} in Section \ref{sec:behaviors}.

\medskip
\subsection{Invariant train tracks for pseudo-Anosov maps} 	
\label{sec:traintrack}
 
 For definitions and basic results on train tracks, see \cite{BestvinaHandel95,PapadopoulosPenner87,FarbMargalit12}.
Let $\psi: S \rightarrow S$ be a pseudo-Anosov homeomorphism defined on a surface $S$ possibly with boundary/punctures. 
When $S$ is a punctured surface, 
we say that $\psi$ is {\it fully punctured} if 
the set of singularities of the unstable foliation for $\psi$ 
is contained in the set of punctures of $S$.   
%We will use this notation in the later section. 

Let $\tau$ be an invariant train track for $\psi$. 
Then $\psi: S \rightarrow S $ induces a map on  $\tau $ to itself  which takes switches (vertices) to themselves. 
Such a map is called the {\it train track map}. 
%\marginal{I may add def of infinitesimal and real branch, and add some words about invariant train track (E)}
By abuse of notations, we denote the train track map on $\tau$ also  by $\psi: \tau \rightarrow \tau$. 
%There are two types of branches of $\tau$; {\it infinitesimal} and {\it real} branches~\cite[Section 3.3]{BestvinaHandel95}.  
Following \cite[Section 3.3]{BestvinaHandel95}, we say that 
a branch $e$ of $\tau$ is {\it real} if 
there exists an integer $m \ge 1$ such that 
$\psi^m(e)$ passes through all branches of $\tau$.  
Otherwise we say that $e$ is {\it infinitesimal}. 
The train track map  $\psi: \tau \rightarrow \tau$ 
induces a finite digraph $\Gamma$ by taking a vertex for each real branch of $\tau$, 
 and then adding $m_{ij}$ directed edges from the $j$th real branch $e_j$ to the $i$th real branch $e_i$, 
 where $m_{i,j}$ is the number of times so that the image $\psi(e_j)$ under the train track map $\psi$ passes through $e_i$ in either direction.

 For the lower bound of $\ell_{\mathcal{C}}(\psi)$, we recall a result by Gadre--Tsai. 
 The following statement is a consequence of Lemma~5.2 in \cite{GadreTsai11} together with the proof of Theorem~5.1 in \cite{GadreTsai11}.

 \begin{prop}
 \label{prop_lowerbound}
 Let $\psi \in \mathrm{Mod}(S_{g,n})$ be a pseudo-Anosov element and 
 let $\tau$ be an invariant train track for $\psi$. 
 Suppose that $r$ is a positive integer such that 
 for any real branch $e$ of $\tau$, 
 $\psi^{r}(e)$ passes through every real branch. 
 If we set $h= r + 24|\chi(S_{g,n})|-8n$, 
 then $\psi^h(e)$ passes through every branch of $\tau$ (including infinitesimal branches).  
 Moreover if we set 
 $$w= h+ 6|\chi(S_{g,n})|-2n = r + 30|\chi(S_{g,n})|-10n \le r + 30|\chi(S_{g,n})|,$$ 
 then we have 
 $$\ell_{\mathcal{C}}(\psi) \ge \frac{1}{w} \ge \frac{1}{r + 30|\chi(S_{g,n})|}.$$
 \end{prop}

%First we review the necessary background from the previous paper of the second author. 
\medskip
\subsection{Fibered cones of the magic manifold} \label{sec:magicmanifold}

\begin{center}
\begin{figure}[t]
\includegraphics[height=4.5cm]{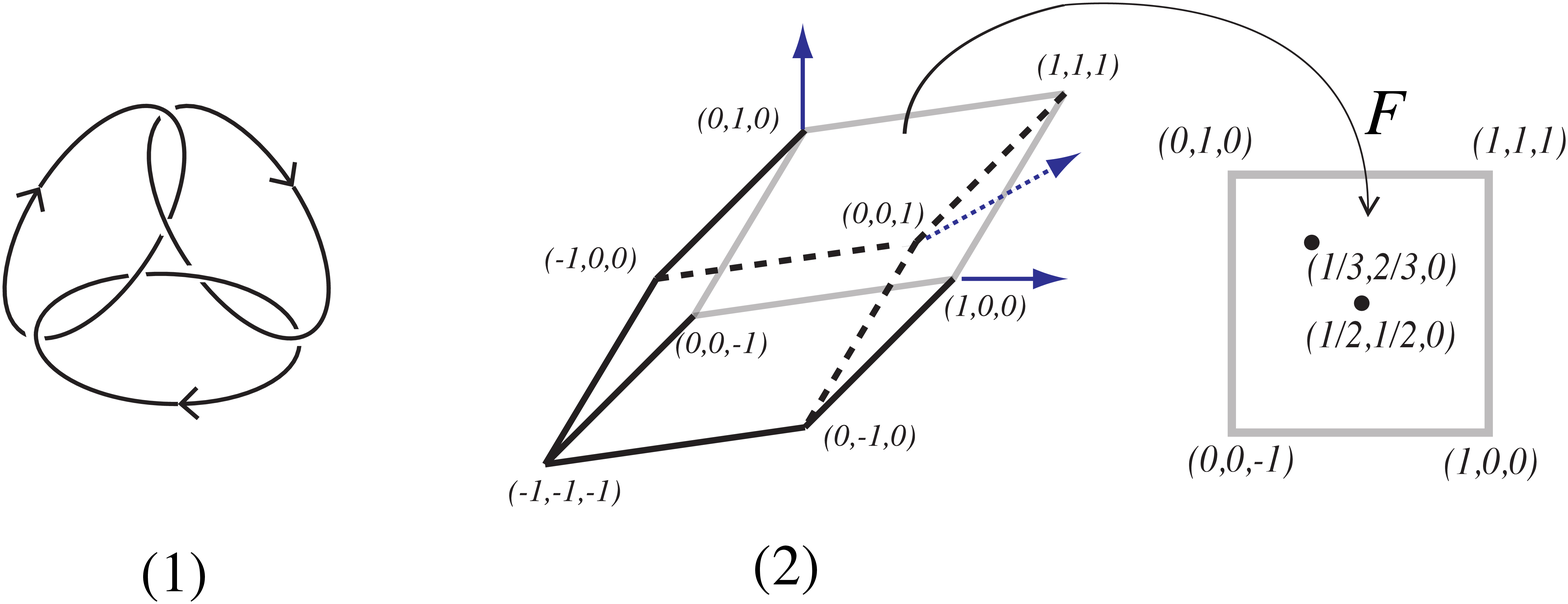}
\caption{(1) $3$ chain link $\mathcal{C}_3$. 
(2) Thurston norm ball of  $N$ and  fibered face $F$.}  
\label{fig_3chainmagic}
\end{figure}
\end{center}

We consider coordinates of integral classes in fibered cones of $N$. 
We assign orientations of the three components  of $\mathcal{C}_3$
as in Figure~\ref{fig_3chainmagic}(1).  
Let $S_{\alpha}$, $S_{\beta}$ and $S_{\gamma}$ be the oriented $2$-punctured disks bounded by 
these components of $\mathcal{C}_3$. 
We set $\alpha= [S_{\alpha}]$, $\beta= [S_{\beta}]$, $\gamma= [S_{\gamma}] \in H_2(N, \partial N; {\Bbb Z}) \simeq H^1(N; {\Bbb Z})$. 
Then $\alpha, \beta, \gamma$ form a basis of $H_2(N, \partial N; {\Bbb Z})$. 
We denote by $(x,y,z)$,  the class $x\alpha+ y \beta+ z \gamma $. 
The Thurston norm ball $B_N$  is the parallelepiped with vertices 
$ \pm \alpha = \pm (1,0,0)$, 
$\pm \beta = \pm (0,1,0)$, 
$\pm \gamma = \pm (0,0,1)$ and 
$\pm (\alpha+ \beta+ \gamma) = \pm (1,1,1)$, 
see Figure~\ref{fig_3chainmagic}(2).  
%(See also \cite[Example~3]{Thurston86}.)

 A symmetry of $\mathcal{C}_3$ tells us that 
every top-dimensional face of $B_N$ is a fibered face. 
Moreover 
all fibered faces of $N$ are permuted transitively by homeomorphisms of  $N$. 
Hence they have the same topological types in their fibers and the same dynamics of their monodromies. 
To study  monodromies of fibrations on $N$, it suffices to pick a particular fibered face, say $F$  with vertices 
$(1,0,0)$, $(1,1,1)$, $(0,1,0)$ and $(0,0,-1)$, 
see Figure~\ref{fig_3chainmagic}(2). 
%\marginal{I added some sentences which explain the issue: boundary vs punctures. 
%Please check that  this is OK or not (E)\\
%This looks OK! I add one more explanation. (S)}
For a primitive integral class $(S, \psi) \in  \mathscr{C}_F$, 
the monodromy $\psi$ is pseudo-Anosov defined on $S$ with boundary components, 
since $\partial N \ne \emptyset$. 
Each connected component of $\partial S$ is a simple closed curve 
which lies on one of the boundary tori of $N$. 
By abusing notations, we often regard 
boundary components of  $S$ as punctures of $S$
by crushing each boundary component to a puncture.
Hence we think of $\psi$ as a pseudo-Anosov map defined on the punctured surface $S$. 
Such ambiguity does not matter for our purpose since the computation of the asymptotic translation lengths of the pseudo-Anosov monodromies on the curve complex will not be affected. 
Under this convention, 
%the fibered face $F$ (and hence each fibered face of $N$) is {\it fully punctured}, 
%i.e., 
one sees that 
for any primitive integral class $(S, \psi)  \in \mathscr{C}_F$, 
the pseudo-Anosov monodromy $\psi$ is fully punctured, see for example \cite{Kin15}. 
%(If there exists a  primitive integral class $(S, \psi)  \in \mathscr{C}_F$ so that $\psi$ is fully punctured, then 
%it follows that $F$ is fully punctured.) 
%for any primitive integral class $(S, \psi) = (S_{(x,y,z)}, \psi_{(x,y,z)}) \in \mathscr{C}_F$, 
%the set of singularities of the unstable foliation of the pseudo-Anosov monodromy $\psi$ 
%is contained in the set of punctures of $S$.   

The open face $int(F)$  is written by 
\begin{eqnarray*}
%\label{equation_OpenFace}
\hspace{1cm}int(F)= \{(x,y,z)\ |\ x+y-z=1, \ x>0,\  y>0,\  x>z,\  y>z\}. 
\end{eqnarray*}
%Let $S_{(x,y,z)}$ be a norm minimizing surface of an integral class $(x,y,z) \in C_{\Delta}$. 
%By \cite{Thurston86}, if an integral class $(x,y,z)$ is in the open cone $int(C_F)$ over $F$, 
%then a norm-minimizing surface $S_{(x,y,z)}$ is a fiber of a fibration on $N$. 
%Furthermore if $(x,y,z) \in int(C_F)$ is primitive, then 
%$S_{(x,y,z)}$ is a connected fiber. 
This implies that 
$(x,y,z) \in \mathscr{C}_F$ 
if and only if 
$x>0$, $y>0$,  $x>z$ and $y>z$. 
The next lemma tells us the topological type of the corresponding fiber $S_{(x,y,z)}$.

\begin{lem}[\cite{KinTakasawa11}]
\label{lem_magicformula}
For a primitive integral class $(x,y,z) \in \mathscr{C}_F$, 
let $|\partial S_{(x,y,z)}|$ denote the number of the boundary components of  $S_{(x,y,z)}$. 
The Thurston norm $\|(x,y,z)\|= |\chi(S_{(x,y,z)})|$ equals $x+y-z$,  and $|\partial S_{(x,y,z)}|$ is given by 
$$|\partial S_{(x,y,z)}|= \gcd(x,y+z) + \gcd(y,z+x)+ \gcd(z, x+y).$$
More precisely, each term in the right-hand side 
expresses the number of boundary components of $S_{(x,y,z)}$ 
which lie on one of the boundary tori of $ N$. 
\end{lem} 

\begin{center}
\begin{figure}[t]
\includegraphics[height=6.5cm]{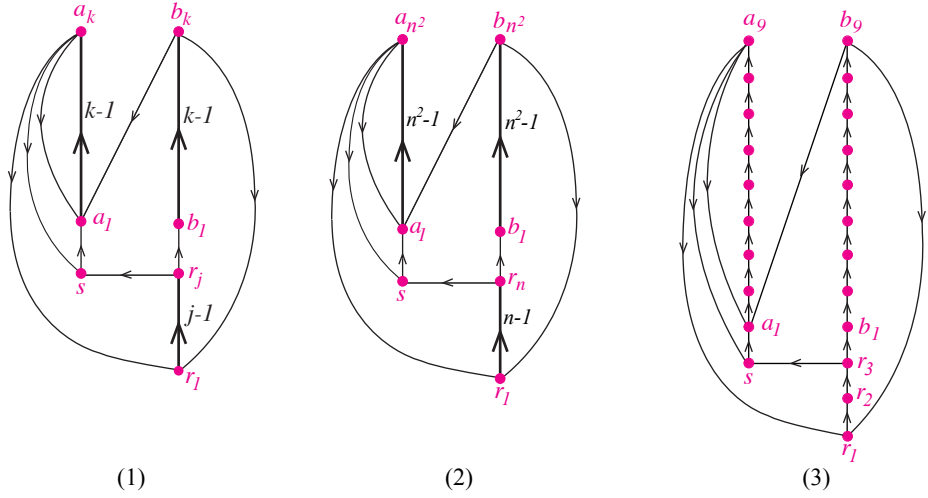}
\caption{Digraphs (1) $\Gamma_{(1,j,k)_+}$,  
(2) $\Gamma_{(1,n,n^2)_+}$  and 
(3) $\Gamma_{(1,3,9)_+}$.} 
\label{fig_digraph_p_original}
\end{figure}
\end{center}

We introduce another coordinate $(i,j,k)_{+} $. 
For  $i,j,k \ge 0$, define
$$(i,j,k)_+= i(1,1,1)+  j(0,1,0)+ k(1,1,0)= (i+k, i+j+k,i). $$
Note that $(1,1,0) \in \mathscr{C}_F$, but $(0,1,0) \notin \mathscr{C}_F$ and $(1,1,1) \notin \mathscr{C}_F$ 
(in fact the two classes lie on  $\partial F$), 
see Figure \ref{fig_3chainmagic}(2). 
We denote by $\overline{(i,j,k)_+}$, the class with the Thurston norm $1$ 
which is projectively equal to $(i,j,k)_+$.

If $i,j,k$ are integers with $i \ge 0$, $j \ge 0$ and $k >0$, then 
 $(i,j,k)_{+} \in \mathscr{C}_F$. 
 If $(i,j,k)_{+}$ is a primitive integral class in $\mathscr{C}_F$, 
 then we let $(S_{(i,j,k)_{+}}, \psi_{(i,j,k)_{+}})$ be the pair of the fiber and its monodromy. 
 In \cite[Section~3]{Kin15}, the second author constructs an invariant train track $\tau= \tau_{(i,j,k)_+}$ 
%(in the sense of Bestvina-Handel) 
and the digraph $\Gamma= \Gamma_{(i,j,k)_+}$ 
of the train track map $\psi= \psi_{(i,j,k)_+}: \tau \rightarrow \tau$ 
for each primitive integral class  $(i,j,k)_+  \in \mathscr{C}_F$.  
Figure~\ref{fig_digraph_p_original}(1) illustrates $\Gamma= \Gamma_{(1,j,k)_+}$ when $i=1$, $j>0$ and $k>0$ 
(see also \cite[Figure~22(4)]{Kin15}). 
The vertices in the left column of $\Gamma$ are denoted by 
$s, a_1,  \cdots, a_k$ from bottom to top; 
vertices in the right column of $\Gamma$ are denoted by 
$r_1, \cdots, r_j, b_1, \cdots, b_k$  from bottom to top. 
(Recall that each vertex of $\Gamma$ corresponds to a real branch of $\tau$.) 
The numbers $j-1$ and $k-1$ near the `thick' edges of $\Gamma$ indicate their lengths of paths.  
For instance, the edge $r_1 \stackrel{j-1}{\longrightarrow} r_j$ from $r_1$ to $r_j$ 
indicates the edge path 
$r_1 \rightarrow  \cdots \rightarrow r_{j-1} \rightarrow r_j$. 
%An edge without number means its length equals $1$. 
See Figure \ref{fig_digraph_p_original}(3) for the concrete example. 
When $j=1$ or $k=1$, the corresponding `thick' edges collapse 
(see Figure~\ref{fig_digraphbraid}).

\medskip
\subsection{Computing the lower bounds} \label{sec:lowerbound}

For fixed positive integers $p$ and $q$, 
%with $\gcd(p,q)=1$,  
we consider the sequence 
$$(1, n^p, n^q)_+ = (1+n^q, 1+ n^p+ n^q,1) \in \mathscr{C}_F$$ 
varying positive integer $n$. 
The integral class $(1, n^p, n^q)_+$ is primitive, 
since $\gcd(1,n^p,n^q) =1$.  
From the formula of the Thurston norm in Lemma \ref{lem_magicformula}, 
it is immediate to see the following lemma. 
See also Figure \ref{fig_3chainmagic}(2). 

\begin{lem}
\label{lem_pq}
Let $\overline{(1,n^p,n^q)_+} $ be the projective class of $(1,n^p,n^q)_+ $. 
\begin{enumerate}
\item[(1)] 
If $p= q$, then 
$\overline{(1,n^p,n^q)_+} \to  (\frac{1}{3}, \frac{2}{3},0)  \in int(F) $ 
%$\overline{(1,n^p,n^q)_+} \to \overline{(0,1,1)_+} = (\frac{1}{3}, \frac{2}{3},0)  \in int(F) $ 
as $n \to \infty$. 

\item[(2)] 
If $p< q$, then 
$\overline{(1,n^p,n^q)_+} \to  (\frac{1}{2}, \frac{1}{2},0)  \in int(F) $ 
%$\overline{(1,n^p,n^q)_+} \to \overline{(0,0,1)_+}=  (\frac{1}{2}, \frac{1}{2},0)  \in int(F) $ 
as $n \to \infty$. 

\item[(3)] 
If $p > q$, then 
$\overline{(1,n^p,n^q)_+} \to  (0,1,0) \in \partial F $ 
%$\overline{(1,n^p,n^q)_+} \to \overline{(0,1,0)_+}= (0,1,0) \in \partial F $ 
as $n \to \infty$. 
\end{enumerate}
\end{lem}

Here we consider 
the following three cases: 
$q< p< 2q$, $p < q \le  2p$ and $2p \le q$. 
%$p < q \le  2p$, $q< p< 2q$ and $2p \le q$. 
%(i.e., we do not deal with the case $2q < p$.)  
%In the rest of Section~\ref{sec:lowerbound}, 
%we assume $p$ and $q$ are fixed  integers in the above three cases. 
We define $$k = k_{p,q} = \begin{cases}
n^q(2n^q+1) & \text{if } q < p < 2q,\\
n^q(2n^p+1) & \text{if } p < q \leq 2p,\\
n^q(2n^{q-p}+1) & \text{if } 2p \leq q.
\end{cases}$$

\begin{prop} 
\label{prop:uniformedgepath} 
For any two vertices $v$, $w$ of 
$\Gamma = \Gamma_{(1,n^p,n^q)_+}$, 
there exists an edge path from $v$ to $w$ of length $k + 2n^p + 3n^q$. 
\end{prop} 
In other words, if we set 
$k'=  k_{p,q} + 2n^p + 3n^q$, then 
for any real branch $v$ of $\tau$, $\psi^{k'}(v)$ passes through every real branch. 
For the proof of Proposition \ref{prop:uniformedgepath}, 
we need some lemmas. 
Recall that $s$ is the bottom vertex in the left column of $\Gamma$. 
Let $v_0$ be the top vertex $a_{n^q}$ in the left column of $\Gamma$ 
(Figure \ref{figure:digraph}). 

%Note that $v_0$ is the unique vertex in the left column of $\Gamma$
%so that there exists an edge  from $v_0$ to $s$. 

%We will show that 
%for any vertex $v$ of $\Gamma$, there exists an edge path 
%from $s$ to $v$ of length exactly $k+n^p+n^q$.  We show this
%via three lemmas. 

\begin{lem} 
\label{lem:leftcolumn} 
For any vertex $v$ in the left column of $\Gamma$, there
  exists an edge path from $s$ to $v$ of length $k$. 
\end{lem} 
\begin{proof} 
We have an edge path 
$s \to a_1 \stackrel{n_q-1}{\longrightarrow} a_{n^q} = v_0$ 
from $s$ to $v_0$ of length $n^	q$. 
For the proof of the lemma, 
it suffices to show that for any vertex $v$ in the left column of
$\Gamma$, there exists an edge path from $v_0$ to $v$ of length $k-n^q$. 
Then the desired path can be obtained from the concatenation of the two paths, 
the path from $s$ to  $v_0$ and the path from $v_0$ to $v$.
Equivalently, we  show that for any $i = 0, \ldots, n^q$,
there exists a cycle based at $v_0$ of length $k - n^q + i$. 

It is easy to find two cycles based at $v_0$ in $\Gamma$ of lengths $n^q$ and  $n^q +1$ 
(see Figure \ref{fig_digraph_p_original}(1)). 
We have another cycle based at $v_0$ in $\Gamma$ of length $n^p + n^q + 1$ as follows: 
$$v_0= a_{n^q} \to r_1   \stackrel{n_p-1}{\longrightarrow} r_{n^p} \to s \to a_1  
\stackrel{n_q-1}{\longrightarrow} a_{n^q}= v_0$$
We show that combining repeated use of
these three cycles is enough to produce the cycles we desire. 
Suppose $q<p<2q$. Then $k-n^q=2n^{2q}$. We now show that for
any $i = 0, \ldots, n^q$, there exist nonnegative integers $a$, $b$, and $c$ such
that 
$$a n^q + b (n^q + 1) + c (n^q + n^p + 1)= 2n^{2q} + i.$$
This is done by setting $c=0$,  $ b=i$ and $a=2n^q-i$.   
%$$a=2n^q-i, \hspace{2mm} b=i,\hspace{2mm} c=0.$$  

Suppose $p < q \leq 2p$. 
%Then $k-n^q = 2n^{2q-p}$. 
Then $k-n^q = 2n^{p+q}$. 
We claim that for
any $i = 0, \ldots, n^q$, there exist nonnegative integers $a$, $b$, and $c$ such
that 
$$a n^q + b (n^q + 1) + c (n^q + n^p + 1)= 2n^{p+q} + i.$$
This can be done by setting 
$$c=\lfloor {i\over n^p+1}\rfloor,\hspace{2mm}  b=i-(n^p+1)\lfloor {i\over n^p+1}\rfloor,\hspace{2mm}  a=2n^p-b-c,$$ 
where $\lfloor  \cdot \rfloor$ is the floor function.  
Here $b$ and $c$ are nonnegative integers by definition, and 
$b$ is the remainder of $i$ divided by $n^p+1$. 
Hence $b$ must be no larger than $n^p$. 
On the other hand $c\leq n^{q-p}$, because $i\leq n^q<n^{q-p}(n^p+1)$. 
Thus $b+c \leq n^p+n^{q-p}\leq 2n^p$, which implies that $a$ is nonnegative. 

Lastly, suppose $2p\leq q$. 
%Then $k-n^q = 2n^{p+q}$. 
Then $k-n^q = 2n^{2q-p}$. 
We claim that for
any $i = 0, \ldots, n^q$, there exist  nonnegative integers $a$, $b$, and $c$ such
that 
$$a n^q + b (n^q + 1) + c (n^q + n^p + 1)= 2n^{2q-p} + i.$$
This can be done by setting 
$$c=\lfloor {i\over n^p+1}\rfloor, \hspace{2mm} b=i-(n^p+1)\lfloor {i\over n^p+1}\rfloor, \hspace{2mm} a=2n^{q-p}-b-c.$$ 
Here $b$ and $c$ are nonnegative integers by definition, 
and $b$ is the remainder of $i$ divided by $n^p+1$. 
Hence $b$ must be no larger than $n^p$. 
On the other hand $c\leq n^{q-p}$, because $i\leq n^q<n^{q-p}(n^p+1)$. 
Thus $b+c\leq n^p+n^{q-p}\leq 2n^{q-p}$, which says that $a$ is nonnegative. 
This finishes the proof. 
\end{proof}

\begin{lem}
\label{lem:leftcolumnext} For any vertex $v$ in the left column of
$\Gamma$ and for any $m \geq 0$, there exists an edge path from $s$ to
$v$ of length $k + m$. 
\end{lem} 
\begin{proof} 
%An easy induction argument. 
Let $v$ be any vertex in the left column of $\Gamma$. 
For any $m \geq 0$, one can find a vertex $v'$ in the left column of $\Gamma$ 
such that there is an edge path from $v'$ to $v$ of length $m$. 
(To see this, use the above cycles based at $v_0$ of lengths $n^q$ and $n^q +1$.) 
Lemma \ref{lem:leftcolumn} tells us that 
there exists an edge path from $s$ to $v'$ of length $k$. 
The concatenation of these edge paths is a desired edge path of length $k+m$. 
\end{proof}

\begin{lem}
\label{lem:rightcolumn} For any vertex $v$ in the right column of
$\Gamma$ and for any $m \geq 0$, there exists an edge path from $s$ to
$v$ of length $k + n^p + n^q + m$. 
\end{lem} 

\begin{proof} 
%Let $v_0$ be as in the proof of Lemma  \ref{lem:leftcolumn}. 
  Let $v$ be an arbitary vertex in the right
  column of $\Gamma$. Then there exists an edge path from $v_0$ to $v$
  of length $\ell$ with $1 \leq \ell \leq n^p + n^q$. 
 To see this, use the path 
 $$v_0= a_{n^q} \to r_1 \stackrel{n_p-1}{\longrightarrow} r_{n^q} \to b_1 \stackrel{n_q-1}{\longrightarrow} b_{n^q} $$
 from $v_0$ to $b_{n^q} $. 
On the other hand, Lemma \ref{lem:leftcolumnext} tells us that 
there exists an edge path from $s$
to $v_0$ of length $k + (n^p + n^q - \ell) + m $. 
Here $(n^p + n^q -\ell) + m$ plays the role of $m$ in Lemma \ref{lem:leftcolumnext}. 
Concatenating these two
paths, one obtains an edge path from $s$ to $v$ of length $k + n^p +
n^q + m$. 
\end{proof} 

By Lemmas \ref{lem:leftcolumnext} and \ref{lem:rightcolumn}, 
we immediately have the following lemma. 

\begin{lem}
\label{lem:keylem}
For any vertex $v$ of $\Gamma$ and for any $m \geq 0$, 
there exists an edge path from $s$ to $v$ of length $k+n^p+n^q +m$. 
\end{lem}

%\begin{prop} 
%\label{prop:uniformedgepath} 
%For any two vertices $v$, $w$ of $\Gamma$, there exists an edge path from $v$ to $w$ of length $k + 2n^p + 3n^q$. 
%\end{prop} 

We are now ready to prove Proposition \ref{prop:uniformedgepath}. 

\begin{proof}[Proof of Proposition \ref{prop:uniformedgepath}] 
Note that for any vertex $v$, there exists an edge path
  from $v$ to $s$ of length $0 \leq \ell  \leq n^p + 2n^q$. 
  To see this, one can use the following edge path of length $n^p + 2n^q$ 
  passing through all vertices of $\Gamma$. 
  $$r_1 \stackrel{n_p-1}{\longrightarrow} r_{n^q} \to b_1 \stackrel{n_q-1}{\longrightarrow} b_{n^q} \to a_1 \stackrel{n_q-1}{\longrightarrow} a_{n^q} \to s.$$
  By Lemma \ref{lem:keylem} there exists an edge path from $s$ to any vertex $w$
  of length exactly $k + (2n^p + 3n^q - \ell)$, since $2n^p + 3n^q -
  \ell \geq n^p + n^q$.  The concatenation of the two paths has length $k + 2n^p + 3n^q$. 
\end{proof} 

Now we are ready to compute the lower bounds.
For real-valued functions $A(x)$ and $B(x)$, we write $A(x) \gtrsim B(x)$ 
if there is a constant $C>0$ independent of $x$ such that
$A(x) \geq C \cdot B(x)$.

\begin{thm} \label{thm:lowerbound} 
The sequence $(1, n^p, n^q)_+$  in $ \mathscr{C}_F$ satisfies 
$$ \ell_\mathcal{C}(\psi_{(1, n^p, n^q)_+}) \gtrsim 
\begin{cases}
1/n^{2q} & \text{if } q < p < 2q,\\
1/n^{p+q} & \text{if } p < q \leq 2p,\\
1/n^{2q-p} & \text{if } 2p  \leq q.
\end{cases} $$ 
\end{thm}

\begin{proof}
By Lemma \ref{lem_magicformula}, 
it is not hard to see that 
$$(k_{p,q} + 2n^p + 3n^q) + 30 |\chi(S_{(1, n^p, n^q)_+})| \asymp 
\begin{cases}
n^{2q} & \text{if } q < p < 2q,\\
n^{p+q} & \text{if } p < q \leq 2p,\\
n^{2q-p} & \text{if } 2p  \leq q.
\end{cases} $$
Then the desired claim follows from 
Propositions  \ref{prop_lowerbound} and \ref{prop:uniformedgepath}.
\end{proof}

\medskip
\subsection{Computing the upper bounds} \label{sec:upperbound}

To prove Theorem \ref{thm:allrational}, we will also compute the upper bound of the asymptotic translation length 
of $\psi_{(1,n^p, n^q)_+}$.

\begin{thm} \label{thm:upperbound}
For any fixed positive integers $p$ and $q$ with $q<p<2q$, 
the sequence  $ (1,n^p, n^q)_+$ of primitive integral classes in $ \mathscr{C}_F$ 
converges projectively to $(0,1,0) \in \partial F$ as $n \to \infty$, and we have 
%we consider the sequence $ (1,n^p, n^q)_+$ of primitive integral classes in $ \mathscr{C}_F$.
%Let $\psi_{(1,n^p,n^q)_+}$ be the corresponding pseudo-Anosov monodromy.
%Then we have
$$\ell_{\C}(\psi_{(1,n^p,n^q)_+}) \leq \frac{4}{n^{2q}}.$$
\end{thm}

The first half of Theorem \ref{thm:upperbound} follows from Lemma \ref{lem_pq}(3). 
For the rest of the proof, we first introduce the dual arcs of real branches of train tracks. 
%Recall that each fibered face of $N$ is fully punctured. 
Consider an invariant train track $\tau$ for the monodromy $\psi$ defined on the fiber $S$ of a fibration on $N$. 
If we think of the surface $S$ with boundary as the punctured surface which is again denoted by $S$ abusing the notation,  
each component of the complement $S \setminus \tau$ of the train track
is a once-punctured ideal polygon, because $\psi$ is fully punctured. 
Consider the cell decomposition of $S$ corresponding to $\tau$.
That is, 0-cells are switches of $\tau$, 1-cells are branches of $\tau$, and 2-cells are ideal polygons
of $S \setminus \tau$.

Given a real branch $v$, the \textit{dual arc} $\alpha_v$ of $v$ is defined to be
the edge of the dual cell complex that connects the punctures in two polygons (possibly the same polygon) 
sharing the real branch $v$ (see Figure \ref{figure:dualarcs}).

\begin{figure}[h]
\centering
\begin{tikzpicture}[scale=0.25]
	\draw (0,0) to [out=70, in=-70] (0,10);
	\draw (0,0) to [out=20, in=160] (10,0);
	\draw[thick, red] (10,0) to [out=110, in=-110] (10,10);
	\draw[thick,black!30!green] (0,10) to [out=-20, in=-160] (10,10);
	
	\draw (10,10) to [out=-50, in=-160] (16,9);
	\draw (16,9) to [out=-90, in=160] (18,5);	
	\draw (18,5) to [out=-160, in=90] (16,1);
	\draw (16,1) to [out=160, in=50] (10,0);
	
	\draw (0,10) to [out=0,in=-10] (0, 14);
	\draw (10,10) to [out=180,in=-170] (10,14);
	\draw (0,14) to [out=0,in=-90] (2,17);
	\draw (10,14) to [out=180,in=-90] (8,17);
	\draw (2,17) to [out=-45,in=-135] (8,17);

	\draw[thick, blue] (5.2,5) to (12.8,5);
	\draw[thick, violet] (5,5.2) to (5, 12.8);
	
	\draw (5,5) circle [radius=.2];
	\draw (13,5) circle [radius=.2];
	\draw (5, 13) circle [radius=.2];
	
	\draw (1, 16) node {$\tau$};
	\draw[,black!30!green] (3,8.2) node {$s$};
	\draw[red] (8.5,7.5) node {$v$};
	\draw[blue] (11,4) node {$\alpha_v$};
	\draw[violet] (6.3,11) node {$\alpha_s$};
	
\end{tikzpicture}
\caption{Cell decomposition, branches, and dual arcs.}
\label{figure:dualarcs}
\end{figure}
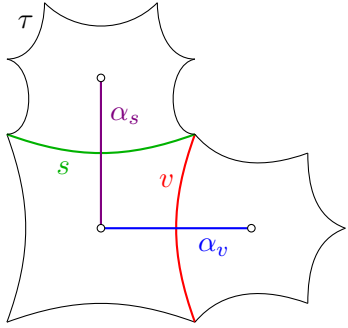

Notice that the dual arc $\alpha_v$ is an essential arc.
In order to see this, consider a rectangle associated with the real branch $v$, contained in a Markov partition
for a pseudo-Anosov homeomorphism which represents $\psi$. 
Then $v$ corresponds to leaves of the unstable foliation and the dual arc $\alpha_v$ corresponds to leaves of 
the stable foliation in this rectangle.
If the dual arc is not essential, then this implies that the real branch $v$ cannot support a positive transverse measure,
which is a contradiction to a property of pseudo-Anosov homeomorphisms. 
%Hence the dual arc $\alpha_v$ is essential.

Readers may notice that  the dual arc associated to a real branch is a general notion for fully punctured pseudo-Anosov homeomorphisms. 
More precisely, 
if $\tau$ is an invariant train track for a fully punctured pseudo-Anosov $\psi$, 
then for a real branch $v$ of $\tau$, 
one can define the dual arc $\alpha_v$ which is essential.

\begin{proof}[Proof of Theorem \ref{thm:upperbound}]
Let $(S, \psi) = (S_{(1,n^p,n^q)_+}, \psi_{(1,n^p,n^q)_+})$ 
be the pair of the fiber and its monodromy for  $(1,n^p,n^q)_+$. 
Let $\Gamma$ be the digraph of the train track $\tau$ for  $(1,n^p,n^q)_+$, 
and let
$\psi_{\ast} : V(\Gamma) \ra V(\Gamma)$ be the induced map, where
$V(\Gamma)$ is the set of vertices of $\Gamma$.
The map $\psi_{\ast}$ can be read off Figure $\ref{figure:digraph}$.

Here is the outline of the proof. 
We will compute the upper bound of the asymptotic translation length $\ell_{\AC}(\psi)$ of $\psi$ 
on the arc and curve complex $\AC(S)$.
Since $\C(S)$ and $\AC(S)$ are quasi-isometric, this gives an upper bound on $C(S)$.
We show that there are distinct vertices $t$ and $v$ in $\Gamma$, 
i.e., distinct real branches  $t$ and $v$ of $\tau$, such that $\psi_{\ast}^{n^{2q}}(t)$ doesn't contain $v$.
Using this fact, we also show that there are disjoint arcs $\beta_t$ and $\alpha_v$ in $\AC(S)$ 
such that $\psi^{n^{2q}}(\beta_t)$ and $\alpha_v$ are disjoint.
This implies that the distance in $\AC(S)$ satisfies 
$d_{\AC}(\beta_t, \psi^{n^{2q}}(\beta_t)) \leq 2$, and we deduce that
$\ell_{\AC}(\psi) \leq \dfrac{2}{n^{2q}}$.

%%%%%%%%%%%%%%%%%%%%%%%%%%%%%%%%%%%%%%%%%%%%%%%%%%%%%%
%
%				Step 1
%
%%%%%%%%%%%%%%%%%%%%%%%%%%%%%%%%%%%%%%%%%%%%%%%%%%%%%%
\vspace{1em}
\textbf{Step 1.} \textit{$\C(S)$ and $\AC(S)$ are quasi-isometric.} 

\vspace{.5em}
Proof of Step 1. Just recall that the inclusion map $\C(S) \ra \AC(S)$ is 2-bilipschitz. 

Hence for the proof of Theorem \ref{thm:upperbound}, 
it is enough to show that the asymptotic translation length $\psi$ on $\AC(S)$ satisfies
$$\ell_{\AC}(\psi) \leq \frac{2}{n^{2q}}.$$

%%%%%%%%%%%%%%%%%%%%%%%%%%%%%%%%%%%%%%%%%%%%%%%%%%%%%%
%
%				Step 2
%
%%%%%%%%%%%%%%%%%%%%%%%%%%%%%%%%%%%%%%%%%%%%%%%%%%%%%%
\vspace{1em}
\textbf{Step 2.} \textit{Let $t $ be the vertex $b_{n^q}$ of $\Gamma$. 
Then $\psi_{\ast}^{n^{2q}}(t)$ doesn't contain all vertices in $\Gamma$.}

\tikzset{->-/.style={decoration={
  markings,
  mark=at position #1 with {\arrow[scale=2]{>}}},postaction={decorate}}}
\begin{figure}[t]
\centering
\begin{tikzpicture}[scale=.6]

	% Blocks
	\fill[red!20] (9.5,.5) rectangle (10.5, 4.5);
	\fill[blue!20] (11.5,.5) rectangle (12.5, 4.5);
	\fill[purple!20] (11.5,-7.5) rectangle (12.5, -3.5);
	\fill[green!20] (11.5,-3.5) rectangle (12.5, .5);

	% edges with arrow
	\foreach \z in {0,10}{
	\foreach \x in {0,2} {
	\foreach \y in {0,1,2,3} {
	\draw[->-=.5] (\x+\z,\y) to (\x+\z,\y+1);
	}}
	\foreach \y in {-1,-2,-3,-4,-5,-6,-7} 
		\draw[->-=.5] (2+\z,\y) to (2+\z,\y+1);
	\draw[->-=.5] (2+\z,0) to (0+\z,0);
	\draw[->-=.5] (2+\z,4) to (0+\z,1);
	\draw[->-=.5] (2+\z,4) to [out=-40, in=40] (2+\z,-7);
	\draw[->-=.5] (0+\z,4) to [out=210, in=180] (2+\z,-7);
	\draw[->-=.5] (0+\z,4) to [out=220, in=120] (0+\z,0);
	\draw[->-=.5] (0+\z,4) to [out=240, in=110] (0+\z,1);
	}

	% vertices a_i, b_i
	\foreach \x in {0,2,10,12} {
	\foreach \y in {0,1,2,3,4} {
	\node at (\x,\y) [circle,fill,inner sep=1.2pt,color=red]{};
	}
	}
	\foreach \y in {-1,-2,-3,-4,-5,-6,-7} {
	\node at (2,\y) [circle,fill,inner sep=1.2pt,color=red]{};
	\node at (12,\y) [circle,fill,inner sep=1.2pt,color=red]{};
	}

	\foreach \y in {-1,-2,-3,-4,-5,-6,-7} {
	\node at (2,\y) [circle,fill,inner sep=1.2pt,color=red]{};
	}
	
	% letters
	\draw[red] (0,-.3) node {\small $s$};
	\draw[red] (0.5,1) node {\small $a_1$};
	\draw[red] (2.5,1) node {\small $b_1$};
	\draw[red] (2.5,2) node {\small $b_2$};
	\draw[red] (2.5,3.2) node {\small $\vdots$};
	\draw[red] (2.6,0) node {\small $r_{n^p}$};
	\draw[red] (2.5,-3.8) node {\small $\vdots$};
	\draw[red] (2.5,-5) node {\small $r_3$};
	\draw[red] (2.5,-6) node {\small $r_2$};
	\draw[red] (2,-7.4) node {\small $r_1$};
	\draw[red] (0,4.4) node {\small $a_{n^q}$};
	\draw[red] (2,4.45) node {\small $b_{n^q}$};
	
	\draw[red] (10,5) node {\small $A$};
	\draw[blue] (12,5) node {\small $B$};
	\draw[green] (13,-.5) node {\small $R_2$};
	\draw[purple] (13,-4.5) node {\small $R_1$};

\end{tikzpicture}
\caption{Digraph $\Gamma_{(1,n^p,n^q)_+}$ (left). 
Digraph $\Gamma_{(1,2^3,2^2)_+}$ with partition $\{A,B, R_1, R_2\}$ (right).}
\label{figure:digraph}
\end{figure}
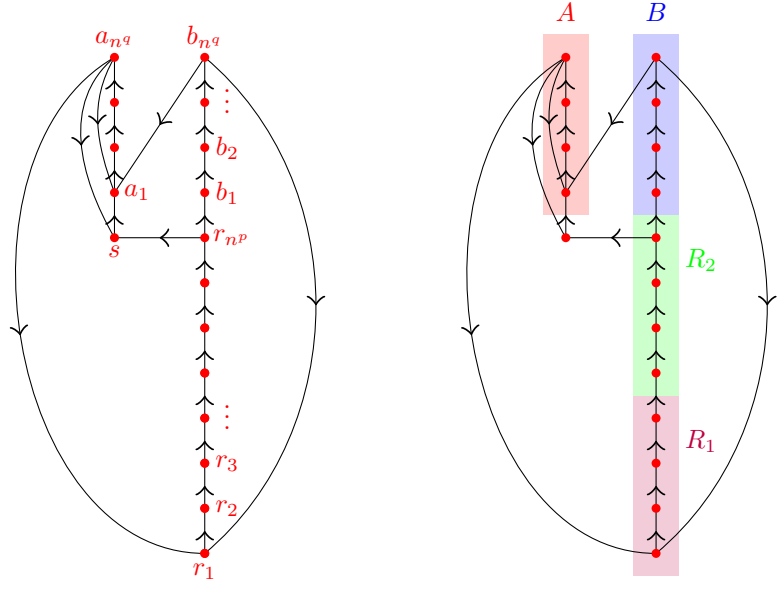

%\marginal{In the caption of figure \ref{figure:digraph}, I add "when $n=2$" since for other $n$ we need  $R_3, \cdots, R_n$}

\vspace{.5em}
%Let $t$ be the vertex $b_{n^q}$ of $\Gamma$.
\noindent
Proof of Step 2.  
We will show that there is a vertex $v$ that is not contained in $\psi_{\ast}^{n^{2q}}(t)$.
Consider the partition $\{ A, B, R_1, R_2, \cdots, R_{n^{p-q}}\}$ of vertices $a_i$, $b_i$, and $r_i$ of $\Gamma$, 
where each partition element 
consists of $n^{q}$ vertices as in Figure \ref{figure:digraph}.
Under the iteration of the  $n^q$th power $\psi_{\ast}^{n^q}$ of $\psi_{\ast}$, one can see that
\begin{align*}
\psi_{\ast}^{n^q}(t) &= \{a_{n^q}, r_{n^q} \},\\
\psi_{\ast}^{2n^q}(t) &= \{a_{n^q}, a_{n^q-1}, r_{n^q}, r_{2n^q} \},\\
\psi_{\ast}^{3n^q}(t) &= \{a_{n^q}, a_{n^q-1}, a_{n^q-2} , r_{n^q}, r_{n^q-1}, r_{2n^q}, r_{3n^q} \},\\
&\vdots
\end{align*}
and that the number of vertices in each partition element, contained in $\psi_{\ast}^{j \cdot n^{q}}(t)$
is increasing by at most one as $j$ increases.
Hence one can see that there are vertices in each $R_k$ $(k= 1, \cdots, n^{p-q})$  
that are not contained in $\psi_{\ast}^{n^{2q}}(t)$.
More precisely, consider $R_1 = \{ r_1, r_2, \cdots, r_{n^{q}} \}$.
One can check that for vertices in $R_1$, the image $\psi_{\ast}^{j \cdot n^{q}}(t)$ contains only
$$\{ r_{n^q}, r_{n^q-1}, \cdots, r_{n^q-j+2} \} \subset R_1$$
for $2 \leq j \leq n^q$.
Therefore $\psi_{\ast}^{n^{2q}}(t)$ does not contain $r_1$, and we may choose $v$ to be $r_1$. 
This completes the proof of Step 2. 
%(There are many other vertices that are not contained in $\psi_{\ast}^{n^{2q}}(t)$.)

%%%%%%%%%%%%%%%%%%%%%%%%%%%%%%%%%%%%%%%%%%%%%%%%%%%%%%
%
%				Step 3
%
%%%%%%%%%%%%%%%%%%%%%%%%%%%%%%%%%%%%%%%%%%%%%%%%%%%%%%
\vspace{1em}
\textbf{Step 3.} \textit{There are distinct arcs $\alpha_v$ and $\beta_t$ in $\AC(S)$ such that 
$\psi^{n^{2q}}(\beta_t)$ and $ \alpha_v$ are disjoint.} 
%the geometric intersection number
%$i(\psi^{n^{2q}}(\beta_t),  \alpha_v)$ equals $0$.}

%\marginal{I modify the sentences in this paragraph for the convenience of proof of the last lemma 4.14}
\vspace{.5em}
Before proving Step 3, 
we first discuss some properties of 
the primitive integral class $(1,j,k)_+$ with $j>0$ and $k > 0$. 
Recall that 
$r_1, \cdots, r_j, b_1, \cdots, b_k$ are  vertices of $\Gamma = \Gamma_{(1,j,k)_+}$ 
which lie on the right column of $\Gamma$ (Figure \ref{fig_digraph_p_original}(1)). 
There is a single ideal polygon $P= P_{(1,j,k)_+}$ containing a single puncture $c_P$ of the fiber $S= S_{(1,j,k)_+}$  
such that the two endpoints of each real branch $b_i$ $(i= 1, \cdots, k)$ are switches (of $\tau$) in the boundary $\partial P$ of $P$, 
see Figure \ref{figure:indealpoly}.
From the construction of $\tau$ in \cite{Kin15}, 
it follows that $\partial P$ consists of periodic branches, i.e., infinitesimal branches, and 
$\psi= \psi_{(1,j,k)_+}$  maps $c_P$ to itself (and hence the ideal polygon $P$ is preserved by $\psi$). 
To see $\psi(c_P)= c_P$, 
we consider  the fiber $S= S_{(i,j,k)_+}$ with boundary. 
(So we now think of the above $c_P$ as a  boundary component of $S$.) 
By using Lemma \ref{lem_magicformula} for the primitive integral class $(1,j,k)_+$, 
we see that 
there is a boundary torus $T$ of $N$ such that 
$c_P$ is the only boundary component of $S$ 
which lies on $T$. 
This implies $c_P$ is preserved by $\psi$.

For the real branch $r_i$ $(i= 1, \cdots, j)$, consider its dual arc $\alpha_{r_i}$. 
Let $c_{r_i}$ and $c'_{r_i}$ be boundary components in  $\partial S$ which are connected by $\alpha_{r_i}$. 
(Possibly $c_{r_i} = c'_{r_i}$.)  
Then there is another boundary torus $T'$ of $N$  on which the both  $c_{r_i}$ and $c'_{r_i}$ lie.
%One sees that the puncture $c_P$  lies on the different boundary torus of $N$. 

\begin{figure}[t]
\centering
\begin{tikzpicture}[scale=.8]

	%\node at (0,0) {$t$};
	
	\draw[blue] (0,0) to [out = 90, in = -170] (1,4);
	\draw[red] (1,4) to [out = 10, in =90] (3,2.05);
	\draw (3,2) circle [radius=.06];
	\draw[red] (3,1.95) to [out = -90, in = 0] (1,-1);
	\draw[red] (1,-1) to [out = 180, in=-90] (0,0);
	
	\draw (1,4) to [out=-170, in=-20] (-2, 3.2);
	\draw (-2, 3.2) to [out=160, in=-20] (-2.5, 3.5);
	\draw (-2, 3.2) to [out=160, in=20] (-2.5, 3);
	\draw (1,4) to [out =10, in = -180] (2, 4.8);
	\draw (1,4) to [out = 10, in = 90] (1.8, 3.5);

	\draw (1.8,3.5) to [out = -90, in = 0] (1.5, 3.3);
	\draw (1.5,3.3) to [out = 0, in = 110] (2, 3.	1);
	
	\draw (0,0) to [out = -90, in = -170] (.8, -.5);
	\draw (.8,-.5) to [out = 10, in = -90] (1.2, 0);
	\draw (1.2,0) to [out = 90, in = 0] (1, 0.3);
	\draw (1,.3) to [out = 0, in = -120] (1.7, 0.5);
	
	\draw[dashed] (1.7, 0.5) to [out = 60, in = -70] (2,3.1);
	
	\draw (0,0) to [out=-90, in = 60] (-.2,-.6);
	
	\draw[dashed] (-.2,-.6) to [out = -120, in = 180] (1.4,-2);
	\draw[dashed] (1.4,-2) to [out = 0, in = -90] (5,2);
	\draw[dashed] (5,2) to [out = 90, in = 0] (2,4.8);
	
	\draw[->] (5.2, 3.8) to (4, 3);
	
	\node at (-.3,1.8) {\color{blue} $t$};
	\node at (3,3.3) {\color{red} $\beta_t$};
	\node at (3.5,1.9) {$c_P$};
	\node at (5.5,4) {$P$};
	\node at (-1, 4) {\large $\tau$};
\end{tikzpicture}
\caption{(A part of train track $\tau$.) 
Ideal polygon $P$, real branch $t = b_{n^q}$, and arc $\beta_t$ based at $c_P$.}
\label{figure:indealpoly}
\end{figure}
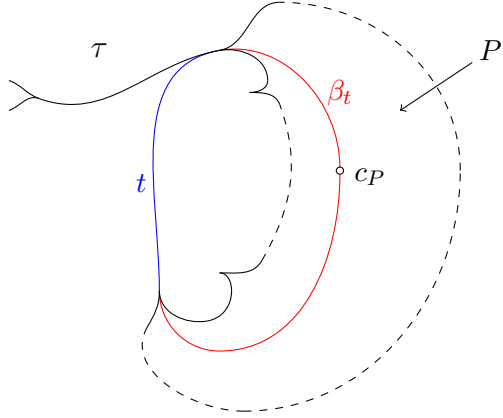

%\begin{lem} \label{lem:distinctarcs}
%We have $\beta_t \neq \alpha_v$ in $\AC(S)$.
%\end{lem}
%\begin{proof}[Proof of Lemma \ref{lem:distinctarcs}]
%%Let $c_v$ and $c'_v$ be punctures which are connected by $\alpha_v$ (possibly $c_v = c'_v$).
%Recall that  $c_v$ and $c'_v$ which are connected by $\alpha_v$
% lie on a boundary torus $T'$ of  $N$, yet 
% $c_P $  lies on the  different boundary torus $T$ of $N$. 
% The arc $\beta_t$ has end points at $c_P$, and hence 
%$\beta_t$ is not homotopic to $\alpha_v$.
%\end{proof}

\medskip
\noindent
Proof of Step 3. 
Consider the primitive integral class $(1,n^p,n^q)_+$ in question. 
The two endpoints of the real branch $t= b_{n^q}$ are switches (of $\tau$) in $\partial P$. 
%The real branch $t$ together with two small arcs contained in $P$ defines an arc $\beta_t$ which connects
%$c_P$  in $P$, see Figure \ref{figure:indealpoly}. 
Join $c_P$ and each endpoint of the real branch $t$ by an arc and then we obtain an arc $\beta_t$ in S (see Figure \ref{figure:indealpoly}). 
Since $t$ is a real branch, one sees that the arc $\beta_t$ is essential.   
Since $\psi$ maps $c_P$ to itself, $\psi^{\ell}(\beta_t)$ is an essential arc based at the same  $c_P$ for each $\ell>0$. 
Moreover  $\psi^{\ell}(\beta_t)$ is not homotopic to $\beta_t$ for each $\ell>0$, 
since $\psi$ is pseudo-Anosov. 
Let us consider the dual arc $\alpha_v$ of $v= r_1$. 
Recall that  $c_v$ and $c'_v$ which are connected by $\alpha_v$
 lie on a boundary torus $T'$ of  $N$, yet 
 $c_P $  lies on the  different boundary torus $T$ of $N$. 
 The arc $\beta_t$ has end points at $c_P$, and hence 
$\beta_t$ is not homotopic to $\alpha_v$.

Now we prove that $\psi^{2q}(\beta_t)$ and $\alpha_v$ are disjoint.
%Consider the image $\psi^{2q}(\beta_t)$ of $\beta_t$. 
The ideal polygon $P$ is preserved by $\psi$,
%so the arc $\psi^{2q}(\beta_t)$ can be homotoped to be contained in the polygon $P$.
and  $\psi^{n^{2q}}(t)$  is carried by $\tau$ since $\tau$ is invariant under $\psi$. 
Moreover, since $\psi^{n^{2q}}(t)$ does not pass through  $v$  by the proof of Step 2, 
it follows that 
$\psi^{n^{2q}}(\beta_t)$ is disjoint from $v$, and hence also disjoint from its dual arc  $\alpha_v$.
This completes the proof of Step 3. 
%Therefore we have $i(\psi^{n^{2q}}(\beta_t),  \alpha_v) = 0$.

%%%%%%%%%%%%%%%%%%%%%%%%%%%%%%%%%%%%%%%%%%%%%%%%%%%%%%
%
%				Step 4
%
%%%%%%%%%%%%%%%%%%%%%%%%%%%%%%%%%%%%%%%%%%%%%%%%%%%%%%
\vspace{1em}
\textbf{Step 4.} \textit{We have
$$ \ell_{\AC}(\psi) \leq \frac{2}{n^{2q}}.$$}

\noindent
Proof of Step 4. 
Clearly $\beta_t$ and $\alpha_v$ are disjoint. 
Since $\psi^{n^{2q}}(\beta_t)$ is an essential arc based at $c_P$, 
we have 
 $\psi^{n^{2q}}(\beta_t) \neq \alpha_v$ in $\AC(S)$ 
 by the same argument as in the proof of Step 3. 
 This together with the fact that the geometric intersection number 
 $i(\psi^{n^{2q}}(\beta_t),  \alpha_v) = 0$ implies that 
$\beta_t$ and $\psi^{n^{2q}}(\beta_t)$ are at most distance 2 in $\AC(S)$, 
i.e., $d_{\AC}(\beta_t, \psi^{n^{2q}}(\beta_t)) \leq 2$. 
By the definition of the asymptotic translation length, it follows that 
$$\ell_{\AC}(\psi) \leq \frac{2}{n^{2q}}.$$
This completes the proof, and 
%Since the inclusion map $\C(S) \ra \AC(S)$ is 2-bilipschitz, 
%we have  
%$$\ell_{C}(\psi) \leq 2 \ell_{\AC}(\psi).$$
we finish the proof of Theorem \ref{thm:upperbound}.
\end{proof}

\begin{thm} \label{thm:upperbound_interior}
For any fixed positive integers $p$ and $q$ with $2p \leq q$, 
the sequence  $ (1,n^p, n^q)_+$ of primitive integral classes in $ \mathscr{C}_F$ 
converges projectively to  $ (\frac{1}{2}, \frac{1}{2},0) \in  int(F)$ as $n \to \infty$, and we have 
$$\ell_{\C}(\psi_{(1,n^p,n^q)_+}) \leq \frac{C}{n^{2q-p}},$$ 
where $C$ is a constant independent on $n$. 
\end{thm}

\begin{proof}
The first half of the claim follows from Lemma \ref{lem_pq}(2). 
For the rest of the proof, 
let $\psi = \psi_{(1,n^p,n^q)_+}$.
Consider the digraph $\Gamma = \Gamma_{(1,n^p,n^q)_+}$ 
and the induced  map $\psi_{\ast}: V(\Gamma) \rightarrow V(\Gamma)$. 
Let $t$ be the vertex $b_{n^q}$ of $\Gamma$. 
By using a similar argument as in Step 2 of the proof of Theorem \ref{thm:upperbound}, 
one can show that 
the set of vertices $\psi_{\ast}^{j \cdot n^{q}}(t)$ is contained in $V(\Gamma) \setminus R$ 
for $j= 1, \cdots, \left\lfloor \frac{n^q-1}{n^p+1} \right\rfloor $, 
where 
$R= \{r_1, r_2, \cdots, r_{n^p}\}$. 
In other words, 
each vertex in $R$ is not contained in $\psi_{\ast}^{j \cdot n^{q}}(t)$ 
for such $j $. 
In particular, if we set $D= D(n)= \left\lfloor \frac{n^q-1}{n^p+1} \right\rfloor $, 
then  $r_1$ is not contained in $\psi_{\ast}^{D  n^{q}}(t)$. 
Then we consider the two arcs $\beta_t$ and $\alpha_v$ as in Step 3 of the proof of Theorem \ref{thm:upperbound}. 
By the same argument, it follows that $\beta_t$, $\alpha_t$ and $\psi^{D  n^{q}}(\beta_t)$ are distinct elements in $\AC(S)$. 
Moreover we have 
$i(\psi^{D n^{q}}(\beta_t), \alpha_v) = 0$ and 
$i(\beta_t, \alpha_v) = 0$. 
Therefore $\beta_t$ and $\psi^{D n^{q}}(\beta_t)$ are at most distance $2$
in $\AC(S)$, and we have
$\displaystyle \ell_{\AC}(\psi) \leq \frac{2}{D n^{q}}$ which implies that 
 $\displaystyle \ell_{\C}(\psi) \leq \dfrac{4}{D n^{q}}$. 
Since $D n^{q} \asymp n^{2q-p}$, 
we finish the proof.
\end{proof}

\medskip
\subsection{The behaviors of asymptotic translation lengths}
\label{sec:behaviors}

\medskip
We prove the following lemma which implies that the upper bound of  Theorem~\ref{thm:newbound} is optimal.

\begin{center}
\begin{figure}[t]
\includegraphics[height=5.2cm]{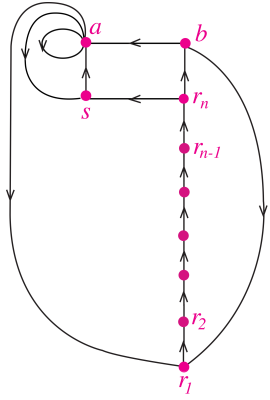}
\caption{Digraph $\Gamma_{(1,n,1)_+}$.} 
\label{fig_digraphbraid}
\end{figure}
\end{center}

\begin{lem} \label{lem:optimalupperbound} 
The sequence 
$(1,n,1)_+$ of primitive integral classes in $  \mathscr{C}_F$ 
converges projectively  to a point in $\partial F$ as $n \to \infty$, and we have 
$$\ell_{\mathcal{C}}(\psi_{(1,n,1)_+}) \asymp \dfrac{1}{|\chi(S_{(1,n,1)_+})|}.$$
\end{lem}

\begin{proof}
The first half of the claim follows from that fact that 
$\overline{(1,n,1)_+} \rightarrow (0,1,0) \in \partial F$ as $n \rightarrow \infty$. 
Since  $|\chi(S_{(1,n,1)_+})| =n+3$, it is enough to prove that 
$\ell_{\mathcal{C}}(\psi_{(1,n,1)_+}) \asymp 1/n$. 
By the digraph $\Gamma= \Gamma_{(1,n,1)_+}$ (see Figure~\ref{fig_digraphbraid}) 
together with Proposition~\ref{prop_lowerbound}, 
it is not hard to see that $\ell_{\mathcal{C}}(\psi_{(1,n,1)_+})  \gtrsim 1/n$.

%Let us set $(S, \psi) = (S_{(1,n,1)_+}, \psi_{(1,n,1)_+})$.  
%We compute the upper bound of $\ell_{\mathcal{C}}(\psi) $. 
%Let $t$ be the vertex $b$ of $\Gamma$. 
 Now we compute the upper bound.
Let $(S, \psi) = (S_{(1,n,1)_+}, \psi_{(1,n,1)_+})$ and 
let $t$ be the vertex $b$ of $\Gamma$. 
We have 
$$\psi_*(t)= \{r_1\},\ \psi^2_*(t)= \{r_2\}, \cdots, \psi^n_*(t)= \{r_n\}.$$
In particular this implies that 
$\psi^n(t)$ does not pass through the real branch $r_1$ of $\tau = \tau_{(1,n,1)_+}$. 
We consider the essential arc $\beta_t$ for $t$ as in the proof of Theorem \ref{thm:upperbound}, 
and consider the dual arc $\alpha_{r_1}$ of $r_1$. 
By the same argument as in the proof of Theorem  \ref{thm:upperbound}, 
one sees that 
the three arcs $\beta_t$, $\psi^n(\beta_t)$ and $\alpha_{r_1}$ are distinct elements in $\AC(S)$. 
Furthermore for the geometric intersection numbers between arcs, 
we have $i(\beta_t,  \alpha_{r_1})=0$ and $i(\psi^n(\beta_t), \alpha_{r_1})=0$. 
Therefore 
$\beta_t$ and $\psi^n(\beta_t)$ are at most distance 2 in $\AC(S)$, 
and we have 
$\ell_{\AC}(\psi) \leq 2/n$, 
which gives the desired upper bound $\ell_{\mathcal{C}}(\psi) \leq 4/n$. 
This completes the proof. 
\end{proof}

Now we are ready to prove the following theorem.

\begin{thm} \label{thm:allrational}
Let $F$ be a fibered face of the magic manifold. 
Then there exist two points  
$\mathfrak{b}_0 \in \partial F$ and $\mathfrak{c}_0 \in int(F)$ 
 which satisfy the following. 
\begin{enumerate}
\item 
For any $r \in \mathbb{Q} \cap [1,2)$, there exists  a sequence $(S_{\alpha_n},  \psi_{\alpha_n})$ of primitive integral classes 
 in $\mathscr{C}_F$ converging projectively to $\mathfrak{b}_0 $ as $n \to \infty$ 
  such that 
$$\ell_{\mathcal{C}}(\psi_{\alpha_n}) \asymp
  \frac{1}{|\chi(S_{\alpha_n})|^{r} }. $$

\item 
For any $r \in \Q \cap [\frac{3}{2}, 2]$, 
there exists a sequence $(S_{\alpha_n},  \psi_{\alpha_n})$ of primitive integral classes  in $\mathscr{C}_F$ 
converging  projectively to $\mathfrak{c}_0$ as $n \rightarrow \infty$
such that 
$$\ell_{\C}(\psi_{\alpha_n}) \asymp \frac{1}{|\chi(S_{\alpha_n})|^r}.$$
In particular, the upper bound in Theorem  \ref{thm:BaikShinWu} is optimal when $d=2$. 
%Let $F$ be a fibered face of the magic manifold. 
%\begin{enumerate}
%\item 
%For any $r \in \mathbb{Q} \cap [1,2)$, there exists  a sequence $(S_{\alpha_n},  \psi_{\alpha_n})$ of primitive integral classes 
% in $\mathscr{C}_F$ converging projectively to a point in  $\partial F$ as $n \to \infty$ 
%  such that 
%$$\ell_{\mathcal{C}}(\psi_{\alpha_n}) \asymp
%  \frac{1}{|\chi(S_{\alpha_n})|^{r} }. $$
%
%\item 
%For any $r \in \Q \cap [\frac{3}{2}, 2]$, 
%there exists a sequence $(S_{\alpha_n},  \psi_{\alpha_n})$ of primitive integral classes  in $\mathscr{C}_F$ 
%converging projectively to a point in  $int(F)$ 
%%and with $|\chi(S_{\alpha_n})| \rightarrow \infty$ 
%as $n \rightarrow \infty$ such that 
%$$\ell_{\C}(\psi_{\alpha_n}) \asymp \frac{1}{|\chi(S_{\alpha_n})|^r}.$$
%In particular the upper bound in Theorem  \ref{thm:BaikShinWu} in the case $r=2$ is optimal. 
\end{enumerate}
\end{thm}

\begin{proof}
Because of the symmetry of the Thurston norm ball $B_N$, 
it suffices to prove the theorem for the  fibered face 
as we picked in Section \ref{sec:magicmanifold}. 
For (1), 
if $1< r <2$,
let $p$ and $q$ be positive integers such that $r= 2q/p$ with  $q < p < 2q$ .
By Lemma \ref{lem_pq}, the sequence$(1,n^p,n^q)_+$ 
converges projectively to $(0,1,0) \in \partial F$. 
By Theorems \ref{thm:lowerbound} and  \ref{thm:upperbound}, 
we have $\ell_{\C}(\psi_{(1,n^p,n^q)_+}) \asymp 1/n^{2q}$. 
Since we have $||(1,n^p,n^q)_+|| \asymp n^p$, it follows that 
$$ \ell_{\C}(\psi_{(1,n^p,n^q)_+}) \asymp \frac{1}{|\chi(S_{(1,n^p,n^q)_+})|^{\frac{2q}{p}}} = \frac{1}{|\chi(S_{(1,n^p,n^q)_+})|^r},$$
where $r = 2q/p \in (1, 2)$. 
If $r=1$, it follows from Lemma \ref{lem:optimalupperbound}.

For (2), 
 if $\frac{3}{2} \le r<2$, let $p$ and $q$ be positive integers such that
$r=2-p/q$ with $2p \leq q$.
By Lemma \ref{lem_pq}, the sequence $(1,n^p,n^q)_+$  converges projectively to 
$(\frac{1}{2}, \frac{1}{2},0) \in int(F)$ as $n \ra \infty$. 
%consider the sequence $(1,n^p,n^q)_+$ with $2p \leq q$.
%By Lemma \ref{lem_pq}, this sequence converges projectively to 
%$(\frac{1}{2}, \frac{1}{2},0) \in int(F)$ as $n \ra \infty$.
By Theorems \ref{thm:lowerbound} and \ref{thm:upperbound_interior}, we have 
$\ell_{\C}(\psi_{\alpha_n}) \asymp 1/n^{2q-p}$.
Since we have  $||(1,n^p,n^q)_+|| \asymp n^q$, it follows that 
$$ \ell_{\C}(\psi_{(1,n^p,n^q)_+}) \asymp \frac{1}{|\chi(S_{(1,n^p,n^q)_+})|^{2-\frac{p}{q}}} = \frac{1}{|\chi(S_{(1,n^p,n^q)_+})|^r},$$
where $r = 2-p/q \in [\frac{3}{2}, 2)$. 
For $r=2$, one can choose a sequence of primitive integral classes 
contained in the intersection between the cone over some compact set $K \subset int(F)$ 
and some 2-dimensional rational subspace of $H^1(M)$. 
(e.g. the sequence $(1,n,n)_+$.) 
Then the sequence satisfies the desired property from  \cite[Corollary 1]{BaikShinWu18}. 
%\marginal{In the proof of thm 4.14, I fix some typo in the last three sentences. 
%please check it. (E)}
%
% what do you mean?? 
%For $r=2$, one can choose a sequence contained in 
%the 2-dimensional subspace over the cone over a compact subset of the fibered face $F$. 
%\marginal{in the proof of thm, what do you mean ``subspace over the cone over a compact subset ..."(E)}

Finally we consider the upper bound in Theorem \ref{thm:BaikShinWu} when $d=2$. 
If $(p,q)= (1,2)$, then 
$$ \ell_{\C}(\psi_{(1,n,n^2)_+}) \asymp  \frac{1}{|\chi(S_{(1,n,n^2)_+})|^{1+\frac{1}{2}}}. $$
Then Theorem \ref{thm:BaikShinWu} 
implies that 
the sequence $(1,n,n^2)_+$ of primitive integral classes can not be contained in any finite union of $2$-dimensional rational subspaces of $H^1(N)$. 
The fibered cone $\mathscr{C}_F$ is a $(2+1)$-dimensional rational subspace of $H^1(N)$. 
%\marginal{I add "The fibered cone $\mathscr{C}_F$ is a $(2+1)$-dimensional rational subspace of $H^1(N)$." (E)}
Thus Theorem \ref{thm:BaikShinWu} is optimal when $d=2$. 
\end{proof}

In light of Theorem  \ref{thm:allrational}(1), we ask the following question. 

\begin{question}
Let $F$ be a fibered face of a compact hyperbolic fibered $3$-manifold. 
Does  there exist a sequence $(S_{\alpha_n},  \psi_{\alpha_n})$ of primitive  integral classes 
 in $\mathscr{C}_F$ converging projectively to  $\partial F$ as $n \to \infty$ 
  such that 
$\ell_{\mathcal{C}}(\psi_{\alpha_n}) \asymp \dfrac{1}{|\chi(S_{\alpha_n})|^2}$?
\end{question}

\medskip

By Theorem \ref{thm:allrational}, 
we immediately have the following corollary.

%%% 			no continuous extension.

\begin{cor}\label{cor:nocontinuousextension} 
%Let $\mathscr{C}$ be a fibered cone of a compact hyperbolic fibered $3$-manifold possibly with boundary. 
%Let $H^1_{\mathrm{prim}}(M; {\Bbb Z})$ be a subset of $H^1(M; {\Bbb Z})$ 
%consisting of primitive integral classes. 
%Then there is no normalization of the asymptotic translation
%  length function 
%  \begin{eqnarray*}
%  \mathscr{C} \cap H^1_{\mathrm{prim}}(M; {\Bbb Z}) &\rightarrow& {\Bbb R}_{\ge 0}
%  \\
%(S_{\alpha}, \psi_{\alpha}) &\rightarrow& \ell_{\C}(\psi_{\alpha})
%  \end{eqnarray*}
%  in terms of the Euler characteristic $\chi(S_{\alpha})$ of the fiber 
%  which admits  a continuous extension to $ \mathscr{C} $

Let $F$ be a fibered face of the magic manifold $N$.
For $\alpha \in F \cap H^1(N;\Q)$, let $(S_{\widetilde{\alpha}}, \psi_{\widetilde{\alpha}})$
be the fiber and pseudo-Anosov monodromy corresponding to the primitive integral class $\widetilde{\alpha}$ lying on 
the ray of $\alpha$ passing through the origin.
Then there is no normalization of the asymptotic translation length function
\begin{align*}
F \cap H^1(N;\Q) & \rightarrow \R_{\geq 0}\\
\alpha &\mapsto \ell_{\C}(\psi_{\widetilde{\alpha}}),
\end{align*}
in terms of the Euler characteristic $\chi(S_{\widetilde{\alpha}})$ which admits a continuous extension on $F$.
\end{cor}

\medskip
\bibliographystyle{alpha} 
\bibliography{fiberedcone}

\begin{thebibliography}{GHKL13}

\bibitem[AT15]{AougabTaylor15}
Tarik Aougab and Samuel~J Taylor.
\newblock Pseudo-anosovs optimizing the ratio of teichm{\"u}ller to curve graph
  translation length.
\newblock {\em In the Tradition of Ahlfors--Bers. VII, in: Contemp. Math},
  696:17--28, 2015.

\bibitem[BH95]{BestvinaHandel95}
M.~Bestvina and M.~Handel.
\newblock Train-tracks for surface homeomorphisms.
\newblock {\em Topology}, 34(1):109--140, 1995.

\bibitem[Bow08]{Bowditch08}
Brian~H. Bowditch.
\newblock Tight geodesics in the curve complex.
\newblock {\em Invent. Math.}, 171(2):281--300, 2008.

\bibitem[BS18]{BaikShin18}
Hyungryul Baik and Hyunshik Shin.
\newblock {Minimal Asymptotic Translation Lengths of Torelli Groups and Pure
  Braid Groups on the Curve Graph}.
\newblock {\em International Mathematics Research Notices}, 12 2018.
\newblock https://doi.org/10.1093/imrn/rny273.

\bibitem[BSW18]{BaikShinWu18}
Hyungryul Baik, Hyunshik Shin, and Chenxi Wu.
\newblock {Upper bound on translation length on curve graph and fibered face}.
\newblock {\em to appear in Indiana U Math J}, 2018.

\bibitem[Cal07]{calegari2007foliations}
Danny Calegari.
\newblock {\em Foliations and the geometry of 3-manifolds}.
\newblock Oxford University Press on Demand, 2007.

\bibitem[FLM08]{FarbLeiningerMargalit08}
Benson Farb, Christopher~J. Leininger, and Dan Margalit.
\newblock The lower central series and pseudo-{A}nosov dilatations.
\newblock {\em Amer. J. Math.}, 130(3):799--827, 2008.

\bibitem[FLP79]{FLP}
A.~Fathi, F.~Laudenbach, and V.~Poenaru.
\newblock {\em Travaux de {T}hurston sur les surfaces}, volume~66 of {\em
  Ast\'erisque}.
\newblock Soci\'et\'e Math\'ematique de France, Paris, 1979.
\newblock S{\'e}minaire Orsay, With an English summary.

\bibitem[FM12]{FarbMargalit12}
Benson Farb and Dan Margalit.
\newblock {\em A primer on mapping class groups}, volume~49 of {\em Princeton
  Mathematical Series}.
\newblock Princeton University Press, Princeton, NJ, 2012.

\bibitem[Fri82a]{Fried82}
David Fried.
\newblock Flow equivalence, hyperbolic systems and a new zeta function for
  flows.
\newblock {\em Comment. Math. Helv.}, 57(2):237--259, 1982.

\bibitem[Fri82b]{fried1982geometry}
David Fried.
\newblock The geometry of cross sections to flows.
\newblock {\em Topology}, 21(4):353--371, 1982.

\bibitem[GHKL13]{GadreHironakaKentLeininger13}
V.~Gadre, E.~Hironaka, R.~P. Kent, IV, and C.~J. Leininger.
\newblock Lipschitz constants to curve complexes.
\newblock {\em Math. Res. Lett.}, 20(4):647--656, 2013.

\bibitem[GT11]{GadreTsai11}
Vaibhav Gadre and Chia-Yen Tsai.
\newblock Minimal pseudo-{A}nosov translation lengths on the complex of curves.
\newblock {\em Geom. Topol.}, 15(3):1297--1312, 2011.

\bibitem[Kin15]{Kin15}
Eiko Kin.
\newblock Dynamics of the monodromies of the fibrations on the magic
  3-manifold.
\newblock {\em New York J. Math.}, 21:547--599, 2015.

\bibitem[KP10]{KorkmazPapadopoulos10}
Mustafa Korkmaz and Athanase Papadopoulos.
\newblock On the arc and curve complex of a surface.
\newblock {\em Math. Proc. Cambridge Philos. Soc.}, 148(3):473--483, 2010.

\bibitem[KS19]{KinShin18}
E.~{Kin} and H.~{Shin}.
\newblock Small asymptotic translation lengths of pseudo-{A}nosov maps on the
  curve complex.
\newblock {\em Groups Geom. Dyn.}, 13(3):883--907, 2019.

\bibitem[KT11]{KinTakasawa11}
Eiko Kin and Mitsuhiko Takasawa.
\newblock Pseudo-{A}nosov braids with small entropy and the magic 3-manifold.
\newblock {\em Comm. Anal. Geom.}, 19(4):705--758, 2011.

\bibitem[Lan19]{Landry19}
Michael Landry.
\newblock {Stable loops and almost transverse surfaces}.
\newblock {\em to appear in Groups, Geometry, and Dynamics}, 2019.

\bibitem[LM18]{LanierMargalit18}
Justin Lanier and Dan Margalit.
\newblock {Normal generators for mapping class groups are abundant}.
\newblock
  \href{http://people.math.gatech.edu/~dmargalit7/papers/normal.pdf}{Preprint},
  2018.

\bibitem[LO97]{LongOertel97}
Darren~D. Long and Ulrich Oertel.
\newblock Hyperbolic surface bundles over the circle.
\newblock In {\em Progress in knot theory and related topics}, volume~56 of
  {\em Travaux en Cours}, pages 121--142. Hermann, Paris, 1997.

\bibitem[Mar]{margp}
Dan Margalit.
\newblock personal communication.

\bibitem[Mat87]{Matsumoto87}
Shigenori Matsumoto.
\newblock Topological entropy and {T}hurston's norm of atoroidal surface
  bundles over the circle.
\newblock {\em J. Fac. Sci. Univ. Tokyo Sect. IA Math.}, 34(3):763--778, 1987.

\bibitem[McM00]{McMullen00}
Curtis~T. McMullen.
\newblock Polynomial invariants for fibered 3-manifolds and {T}eichm\"uller
  geodesics for foliations.
\newblock {\em Ann. Sci. \'Ecole Norm. Sup. (4)}, 33(4):519--560, 2000.

\bibitem[MM99]{MasurMinsky99}
Howard~A. Masur and Yair~N. Minsky.
\newblock Geometry of the complex of curves. {I}. {H}yperbolicity.
\newblock {\em Invent. Math.}, 138(1):103--149, 1999.

\bibitem[MM00]{MasurMinsky00}
H.~A. Masur and Y.~N. Minsky.
\newblock Geometry of the complex of curves. {II}. {H}ierarchical structure.
\newblock {\em Geom. Funct. Anal.}, 10(4):902--974, 2000.

\bibitem[Mos91]{mosher1991surfaces}
Lee Mosher.
\newblock Surfaces and branched surfaces transverse to pseudo-anosov flows on
  3-manifolds.
\newblock {\em Journal of Differential Geometry}, 34(1):1--36, 1991.

\bibitem[PP87]{PapadopoulosPenner87}
Athanase Papadopoulos and Robert~C. Penner.
\newblock A characterization of pseudo-{A}nosov foliations.
\newblock {\em Pacific J. Math.}, 130(2):359--377, 1987.

\bibitem[Str18]{Strenner18}
Bal\'{a}zs Strenner.
\newblock {Fibrations of 3-manifolds and asymptotic translation length in the
  arc complex}.
\newblock {\em ArXiv e-prints}, 2018.

\bibitem[Thu86]{Thurston86}
William~P. Thurston.
\newblock A norm for the homology of {$3$}-manifolds.
\newblock {\em Mem. Amer. Math. Soc.}, 59(339):i--vi and 99--130, 1986.

\bibitem[Thu14]{thurston2014entropy}
William Thurston.
\newblock Entropy in dimension one.
\newblock {\em arXiv preprint arXiv:1402.2008}, 2014.

\bibitem[Val14]{Valdivia14}
Aaron~D. Valdivia.
\newblock Asymptotic translation length in the curve complex.
\newblock {\em New York J. Math.}, 20:989--999, 2014.

\bibitem[Val17]{Valdivia17}
Aaron~D. Valdivia.
\newblock Lipschitz constants to curve complexes for punctured surfaces.
\newblock {\em Topology Appl.}, 216:137--145, 2017.

\end{thebibliography}

\end{document}